\documentclass[a4paper,12pt]{amsart}
\usepackage{amsmath}
\usepackage{amssymb}
\usepackage{xcolor}
\usepackage{amsthm}
\usepackage{amsfonts}
\usepackage[francais, english]{babel}

\DeclareMathSymbol{\subsetneqq}{\mathbin}{AMSb}{36}

\newcommand{\R}{\mathbb{R}}

\newcommand{\N}{\mathbb{N}}

\newcommand{\C}{\mathbb{C}}

\newcommand{\beq}{\begin{eqnarray}}
\newcommand{\eeq}{\end{eqnarray}}
\newcommand{\bq}{\begin{equation}}
\newcommand{\eq}{\end{equation}}
\newcommand{\beqn}{\begin{eqnarray*}}
\newcommand{\eeqn}{\end{eqnarray*}}
\newcommand{\bex}{\begin{exo}}
\newcommand{\eex}{\end{exo}}
\newcommand{\ben}{\begin{enumerate}}
\newcommand{\een}{\end{enumerate}}

\newtheorem{th1}{{\bf Theorem}}[section]
\newtheorem{thm}[th1]{{\bf Theorem}}
\newtheorem{lem}[th1]{{\bf Lemma}}
\newtheorem{prop}[th1]{{\bf Proposition}}

\newtheorem{defi}[th1]{\bf Definition}

\author[R. Ghanmi and T. Saanouni]{R. Ghanmi and T. Saanouni}
\address{University Tunis El Manar,
Faculty of Sciences of Tunis, LR03ES04 partial differential equations and applications, 2092 Tunis, Tunisia.}
\email{\sl Tarek.saanouni@ipeiem.rnu.tn}
\email{\sl ghanmiradhia@gmail.com}


\title[Fourth-order coupled NLS]{On defocusing fourth-order coupled nonlinear Schr\"odinger equations}

\date{\today}
\begin{document}
\begin{abstract}
We investigate the initial value problem for some defocusing coupled nonlinear fourth-order Schr\"odinger equations. Global well-posedness and scattering in the energy space are obtained.
\end{abstract}
\maketitle
\tableofcontents
\vspace{ 1\baselineskip}
\renewcommand{\theequation}{\thesection.\arabic{equation}}
\section{Introduction}
We consider the initial value problem for a defocusign fourth-order Schr\"odinger system with power-type nonlinearities
\begin{equation}
\left\{
\begin{array}{ll}
i\frac{\partial }{\partial t}u_j +\Delta^2 u_j+ \displaystyle\sum_{k=1}^{m}a_{jk}|u_k|^p|u_j|^{p-2}u_j=0 ;\\
u_j(0,x)= \psi_{j}(x),
\label{S}
\end{array}
\right.
\end{equation}
where $u_j: \R \times \R^N \to \C$ for $j\in[1,m]$ and $a_{jk} =a_{kj}$ are positive real numbers.\\
Fourth-order Schr\"odinger equations have been introduced by Karpman \cite{Karpman} and Karpman-Shagalov \cite{Karpman 1} to take into account the role of small fourth-order dispersion terms in the propagation of intense laser beams in a bulk medium with Kerr nonlinearity.\\
The m-component coupled nonlinear Schr\"odinger system with power-type nonlinearities
\begin{eqnarray*}
i\frac{\partial }{\partial t}u_j +\Delta u_j= \pm \displaystyle\sum_{k=1}^{m}a_{jk}|u_k|^p|u_j|^{p-2}u_j ,
\end{eqnarray*}
arises in many physical problems. This models physical systems in which the field has more than one component. For example, in optical fibers and waveguides, the propagating electric field has two components that are transverse to the direction of propagation. Readers are referred to various other works \cite{Hasegawa, Zakharov} for the derivation and applications of this system.\\
A solution ${\bf u}:= (u_1,...,u_m)$ to \eqref{S} formally satisfies respectively conservation of the mass and the energy
\begin{gather*}
M(u_j):= \displaystyle\int_{\R^N}|u_j(x,t)|^2\,dx = M(\psi_{j});\\
E({\bf u}(t)):= \frac{1}{2}\displaystyle \sum_{j=1}^{m}\displaystyle\int_{\R^N}|\Delta u_j|^2\,dx + \frac{1}{2p}\displaystyle \sum_{j,k=1}^{m}a_{jk}\displaystyle \int_{\R^N} |u_j(x,t)|^p |u_k(x,t)|^p\,dx = E({\bf u}(0)).
\end{gather*}

Before going further let us recall some historic facts about this problem. 
The model case given by a pure power nonlinearity is of particular interest. The question of well-posedness in the energy space $H^2$ was widely investigated. We denote for $p>1$ the fourth-order Schr\"odinger problem
$$(NLS)_p\quad i\partial_t u+\Delta^2u\pm u|u|^{p-1}=0,\quad u:{\mathbb R}\times{\mathbb R}^N\rightarrow{\mathbb C}.$$
This equation satisfies a scaling invariance. In fact, if $u$ is a solution to $(NLS)_p$ with data $u_0$, then
$ u_\lambda:=\lambda^{\frac4{p-1}}u(\lambda^4\, .\,,\lambda\, .\,)$
is a solution to $(NLS)_p$ with data $\lambda^{\frac4{p-1}}u_0(\lambda\,.\,).$
For $s_c:=\frac N2-\frac4{p-1}$, the space $\dot H^{s_c}$ whose norm is invariant under the dilatation $u\mapsto u_{\lambda}$ is relevant in this theory. When $s_c=2$ which is the energy critical case, the critical power is $p_c:=\frac{N+4}{N-4}$, $N\geq 5$. Pausader \cite{Pausader} established global well-posedness in the defocusing subcritical case, namely $1< p < p_c$. Moreover, he established global well-posedness and scattering for radial data in the defocusing critical case, namely $p=p_c$. The same result in the critical case without radial condition was obtained by Miao, Xu and Zhao \cite{Miao}, for $N\geq 9$. The focusing case was treated by the last authors in \cite{Miao 1}. They obtained results similar to one proved by Kenig and Merle \cite{Merle} in the classical Schr\"odinger case. See \cite{ts} in the case of exponential nonlinearity.\\

In this note, we combine in some meaning the two problems $(NLS)_p$ and $(CNLS)_p.$ Thus, we have to overcome two difficulties. The first one is the presence of bilaplacian in Schr\"odinger operator and the second is the coupled nonlinearities. We claim that the critical exponent for local well-posedness of \eqref{S} in the energy space is $p=\frac{N}{N - 4}.$ But some technical difficulties yield the condition $4\leq N \leq 6$.\\
It is the purpose of this manusrcipt to obtain global well-posedness and scattering of \eqref{S} via Morawetz estimate.\\

The rest of the paper is organized as follows. The next section contains the main results and some technical tools needed in the sequel. The third and fourth sections are devoted to prove well-posedness of \eqref{S}. In section five, scattering is established. In appendix, we give a proof of Morawetz estimate and a blow-up criterion.
\\
We define the product space
$$H:={H^2({\R^N})\times...\times H^2({\R^N})}=[H^2({\R^N})]^m$$
where $H^2(\R^N)$ is the usual Sobolev space endowed with the complete norm 
$$ \|u\|_{H^2(\R^N)} := \Big(\|u\|_{L^2(\R^N)}^2 + \|\Delta u\|_{L^2(\R^N)}^2\Big)^\frac12.$$
We denote the real numbers 
 $$p_*:=1+\frac4N\quad\mbox{ and }\quad p^*:=\left\{
\begin{array}{ll}
\frac{N}{N-4}\quad\mbox{if}\quad N>4;\\
\infty\quad\mbox{if}\quad N=4.
\end{array}
\right.$$
We mention that $C$ will denote a
constant which may vary
from line to line and if $A$ and $B$ are nonnegative real numbers, $A\lesssim B$  means that $A\leq CB$. For $1\leq r\leq\infty$ and $(s,T)\in [1,\infty)\times (0,\infty)$, we denote the Lebesgue space $L^r:=L^r({\mathbb R}^N)$ with the usual norm $\|\,.\,\|_r:=\|\,.\,\|_{L^r}$, $\|\,.\,\|:=\|\,.\,\|_2$ and 
$$\|u\|_{L_T^s(L^r)}:=\Big(\int_{0}^{T}\|u(t)\|_r^s\,dt\Big)^{\frac{1}{s}},\quad \|u\|_{L^s(L^r)}:=\Big(\int_{0}^{+\infty}\|u(t)\|_r^s\,dt\Big)^{\frac{1}{s}}.$$
For simplicity, we denote the usual Sobolev Space $W^{s,p}:=W^{s,p}({\mathbb R}^N)$ and $H^s:=W^{s,2}$. If $X$ is an abstract space $C_T(X):=C([0,T],X)$ stands for the set of continuous functions valued in $X$ and $X_{rd}$ is the set of radial elements in $X$, moreover for an eventual solution to \eqref{S}, we denote $T^*>0$ it's lifespan.

\section{Background Material}
In what follows, we give the main results and collect some estimates needed in the sequel.
\subsection{Main results}
First, local well-posedness of the fourth-order Schr\"odinger problem \eqref{S} is obtained.
\begin{thm}\label{existence}
Let $4\leq N\leq 6$, $ 1< p \leq p^*$ and $ \Psi \in H$. Then, there exist $T^*>0$ and a unique maximal solution to \eqref{S}, 
$$ {\bf u} \in C ([0, T^*), H).$$ Moreover,
\begin{enumerate}
\item ${\bf u}\in \big(L^{\frac{8p}{N(p-1)}}([0, T^*], W^{2,2p})\big)^{(m)};$
\item ${\bf u}$ satisfies conservation of the energy and the mass;
\item $T^*=\infty$ in the subcritical case $(1<p<p^*)$.
\end{enumerate}
\end{thm}
In the critical case, global existence for small data holds in the energy space.
\begin{thm}\label{glb}
Let $4<N\leq6$ and $p=p^*.$ There exists $\epsilon_0>0$ such that if $\Psi:=(\psi_1,...,\psi_m) \in H$ satisfies $\xi(\Psi):= \displaystyle \sum_{j=1}^m\displaystyle\int_{\R^N}|\Delta \psi_j|^2\,dx\leq \epsilon_0$, the system \eqref{S} possesses a unique solution ${\bf u}\in C(\R, H)$.
\end{thm}
Second, the system \eqref{S} scatters in the energy space. Indeed, we show that every global solution of \eqref{S} is asymptotic, as $t\to\pm\infty,$ to a solution of the associated linear Schr\"odinger system.
\begin{thm}
Let $4\leq N\leq 6$ and $ p_*< p< p^*.$ Take ${\bf u}\in C(\R, H),$ a global solution to \eqref{S}. Then
$${\bf u}\in \big(L^{\frac{8p}{N(p - 1)}}(\R, W^{2, 2p})\big)^{(m)}$$
and there exists $\Psi:=(\psi_1,...,\psi_m)\in H$ such that 
$$\lim_{t\longrightarrow\pm\infty}\|{\bf u}(t)-(e^{it\Delta^2}\psi_1,...,e^{it\Delta^2}\psi_m)\|_{H^2}=0.$$
\end{thm}
In the next subsection, we give some standard estimates needed in the paper.
\subsection{Tools}
We start with some properties of the free fourth-order Schr\"odinger kernel.
\begin{prop}\label{fre}
Denoting the free operator associated to the fourth-order fractional Schr\"odinger equation
$$e^{it\Delta^2}u_0:=\mathcal F^{-1}(e^{it|y|^{4}})*u_0,$$
yields
\begin{enumerate}
\item
$e^{it\Delta^2}u_0$ is the solution to the linear problem associated to $(NLS)_p$;
\item
$e^{it\Delta^2}u_0 \mp i\int_0^te^{i(t-s)\Delta^2}u|u|^{p-1}\,ds$ is the solution to the problem $(NLS)_p$;
\item
$(e^{it\Delta^2})^*=e^{-it\Delta^2}$;
\item
$e^{it\Delta^2}$ is an isometry of $L^2$.
\end{enumerate}
\end{prop}
Now, we give the so-called Strichartz estimate \cite{Pausader}.
\begin{defi}
A pair $(q,r)$ of positive real numbers is said to be admissible if 
$$2\leq q,r\leq \infty,\quad (q,r,N) \neq(2, \infty, 4)\quad \mbox{and} \quad \frac{4}{q} = N\Big(\frac{1}{2} - \frac{1}{r}\Big).$$
\end{defi}
\begin{prop}Let two admissible pairs $(q,r),\, (a,b)$ and $T>0.$ Then, there exists a positive real number $C$ such that
\begin{gather}
\|u\|_{L_T^q(W^{2,r})}\leq C \Big( \|u_0\|_{H^2} + \|i\frac{\partial}{\partial t} u+ \Delta^2 u \|_{L_T^{ a^\prime}(W^{2,b^\prime})}\Big);\label{S1}\\ \|\Delta u\|_{L^q_T(L^r)}\leq C \Big(\|\Delta u_0\|_{L^2} + \|i\frac{\partial}{\partial t} u+ \Delta^2 u\|_{L^2_T(\dot W^{1,\frac{2N}{N +2}})}\Big)\label{S2}.
\end{gather}
\end{prop}
The following Morawetz estimate is essential in proving scattering.  
\begin{prop}\label{prop2''}
Let $4\leq N\leq6$, $1<p\leq p^*$ and ${\bf u}\in C(I,H)$ the solution to \eqref{S}. Then,
\begin{enumerate}
\item
if $N>5$, 
\begin{equation}\label{mrwtz1}
\sum_{j=1}^m\int_I\int_{\R^N\times\R^N}\frac{|u_j(t,x)|^2|u_j(t,y)|^2}{|x-y|^5}dxdydt\lesssim_u1;
\end{equation}
\item
if $N=5$, 
\begin{equation}\label{mrwtz2}
\sum_{j=1}^m\int_I\int_{\R^5}|u_j(t,x)|^4dxdt\lesssim_u1.\end{equation}
\end{enumerate}
\end{prop}
For the the reader convenience, a proof which follows as in \cite{Miao 1,Miao 2}, is given in appendix. Let us gather some useful Sobolev embeddings \cite{Adams}.
\begin{prop}\label{injection}
The continuous injections hold
\begin{enumerate}
\item $ W^{s,p}(\R^N)\hookrightarrow L^q(\R^N)$ whenever
$1<p<q<\infty, \quad s>0\quad \mbox{and}\quad \frac{1}{p}\leq \frac{1}{q} + \frac {s}{N};$
\item $W^{s,p_1}(\R^N)\hookrightarrow W^{s - N(\frac{1}{p_1} - \frac{1}{p_2}),p_2}(\R^N)$ whenever $1\leq p_1\leq p_2 \leq \infty.$
\end{enumerate}
\end{prop}
We close this subsection with some absorption result \cite{Tao}.
\begin{lem} \label{Bootstrap}{({Bootstrap Lemma})} Let $T>0$ and $X\in C([0,  T], \R_+)$ such that
$$ X\leq a + b X^{\theta}\quad on \quad [0,T],$$
where  $a,\, b>0,\, \theta>1,\, a<(1 - \frac{1}{\theta})\frac{1}{(\theta b)^{\frac{1}{\theta}}}$ and $X(0)\leq \frac{1}{(\theta b)^{\frac{1}{\theta -1}}}.$ Then
$$X\leq \frac{\theta}{\theta - 1}a\quad on \quad [0, T].$$
\end{lem}
\section{Local well-posedness}
This section is devoted to prove Theorem \ref{existence}. The proof contains three steps. First we prove existence of a local solution to \eqref{S}, second we show uniqueness and finally we establish global existence in subcritical case.
\subsection{Local existence}
We use a standard fixed point argument. For $T>0,$ we denote the space
$$E_T:=\big(C([0,T],H^2)\cap L^{\frac{8p}{N(p-1)}}([0, T], W^{2,2p})\big)^{(m)}$$ with the complete norm $$\|{\bf u}\|_T:=\displaystyle\sum_{j=1}^m\Big(\|u_j\|_{L_T^\infty(H^2)}+\| u_j\|_{L^{\frac{8p}{N(p-1)}}_T( W^{2,2p})}\Big).$$
Define the function
$$\phi({\bf u})(t) := T(t){\Psi} - i \displaystyle\sum_{j,k=1}^{m}\displaystyle\int_0^tT(t-s)\big(|u_k|^p|u_j|^{p-2}u_j(s)\big)\,ds,$$
where $T(t){\Psi} := (e^{it\Delta^2}\psi_{1},...,e^{it\Delta^2}\psi_{m}).$ We prove the existence of some small $T, R >0$ such that $\phi$ is a contraction on the ball $ B_T(R)$ whith center zero and radius $R.$ Take ${\bf u}, {\bf v}\in E_T$ applying the Strichartz estimate \eqref{S1}, we get
$$\|\phi({\bf u}) - \phi({\bf v})\|_T\lesssim \displaystyle\sum_{j, k=1}^{m}\Big\||u_k|^p|u_j|^{p-2}u_j -  |v_k|^p |v_j|^{p-2}v_j\Big\|_{L^{\frac{8p}{p(8-N) + N}}(W^{2,{\frac{2p}{2p-1}}})}.$$
To derive the contraction, consider the function
$$ f_{j,k}: \C^m\rightarrow \C,\, (u_1,...,u_m)\mapsto |u_k|^p|u_j|^{p-2}u_j.$$
With the mean value Theorem
\begin{equation}\label{H1}
|f_{j,k}({\bf u})-f_{j,k}({\bf v})|\lesssim\max\{ |u_k|^{p - 1}|u_j|^{p - 1}+{|u_k|^{p}|u_j|^{p-2}}, |v_k|^p|v_j|^{p - 2}+{|v_k|^{p - 1}|v_j|^{p - 1}}\}|{\bf u} - { \bf v}|.
\end{equation}
Using H\"older inequality, Sobolev embedding and denoting the quantity 
$$ (\mathcal{I}):=\| f_{j,k}({\bf u})-f_{j,k}({\bf v})\|_{L_T^{\frac{8p}{p(8-N) + N}}(L^{\frac{2p}{2p-1}})},$$ we compute via a symmetry argument
\begin{eqnarray*}
(\mathcal{I})
&\lesssim &\big\| \big(|u_k|^{p - 1}|u_j|^{p - 1} +|u_k|^p|u_j|^{p - 2}\big)|{\bf u} - { \bf v}|\big\|_{L_T^{\frac{8p}{p(8-N) + N}}(L^{\frac{2p}{2p-1}})} \\
&\lesssim&\|{\bf u} - { \bf v}\|_{L_T^{\frac{8p}{N(p-1)}}(L^{2p})} \big\| |u_k|^{p-1}|u_j|^{p -1} + |u_k|^{p}|u_j|^{p-2}  \big\|_{L_T^{\frac{8p}{8p - 2N(p-1)}}(L^{\frac{p}{p-1}})}\\
&\lesssim&T^{\frac{8p - 2N(p-1)}{8p}} \|{\bf u} - { \bf v}\|_{L_T^{\frac{8p}{N(p-1)}}(L^{2p})}\big\| |u_k|^{p-1}|u_j|^{p-1}
+ |u_k|^{p}|u_j|^{p-2}  \big\|_{L_T^\infty(L^{\frac{p}{p-1}})} \\
&\lesssim& T^{\frac{8p - 2N(p-1)}{8p}} \|{\bf u} - { \bf v}\|_{L_T^{\frac{8p}{N(p-1)}}(L^{2p})}\Big(\|u_k^{p-1}\|_{L_T^\infty(L^{\frac{2p}{p-1}})}\|u_j^{p-1}\|_{L_T^\infty(L^{\frac{2p}{p-1}})}\\ 
&+& \|u_k^{p}\|_{L_T^\infty(L^2)}\|u_j^{p-2}\|_{L_T^\infty(L^{\frac{2p}{p-2}})} \Big)\\
&\lesssim& T^{\frac{8p - 2N(p-1)}{8p}}  \|{\bf u} - { \bf v}\|_{L_T^{\frac{8p}{N(p-1)}}(L^{2p})}\Big(\|u_k\|_{L_T^\infty(L^{2p})}^{p-1}\|u_j\|_{L_T^\infty(L^{2p})}^{p-1} + \|u_k\|_{L_T^\infty(L^{2p})}^p\|u_j\|_{L_T^\infty(L^{2p})}^{p-2} \Big)\\
&\lesssim& T^{\frac{8p - 2N(p-1)}{8p}}  \|{\bf u} - { \bf v}\|_{L_T^{\frac{8p}{N(p-1)}}(L^{2p})} \Big(\|u_k\|_{L_T^\infty(H^2)}^{p-1}\|u_j\|_{L_T^\infty(H^2)}^{p-1} 
+ \|u_k\|_{L_T^\infty(H^2)}^p\|u_j\|_{L_T^\infty(H^2)}^{p-2} \Big).
\end{eqnarray*}
Then 
\begin{eqnarray*}
\displaystyle\sum_{k,j=1}^m\| f_{j,k}({\bf u})-f_{j,k}({\bf v})\|_{L_T^{\frac{8p}{p(8-N) + N}}(L^{\frac{2p}{2p-1}})}
&\lesssim & T^{\frac{8p - 2N(p-1)}{8p}} R^{2p-2}\|{\bf u} - {\bf v}\|_{T}.
\end{eqnarray*}
It remains to estimate the quantity $$\big\|\Delta \big(f_{j,k}({\bf u}) - f_{j,k}({\bf v})\big)\big\|_{L_T^{\frac{8p}{p(8-N) + N}}(L^{\frac{2p}{2p-1}})}.$$
Write
\begin{eqnarray*}
\partial_i^2\Big((f_{j,k}({\bf u}) - f_{j,k}({\bf v})\Big)
&=&\partial_i \Big({u}_i (f_{j,k})_i({\bf u}) - {v}_i(f_{j,k})_i({\bf v})\Big)\\
&=&{\bf u}_{ii}(f_{j,k})_i({\bf u}) - {\bf v}_{ii}(f_{j,k})_i({\bf v})  + {u}_i ^2(f_{j,k})_{ii}({\bf u}) - {v}_i ^2(f_{j,k})_{ii}({\bf v})\\
& =&({\bf u} - {\bf v})_{ii}(f_{j,k})_i({\bf u}) +  {\bf v}_{ii}\Big((f_{j,k})_i({\bf u}) - (f_{j,k})_i({\bf v})\Big) \\
&+&\Big({u}_i^2 - {v}_i^2\Big)(f_{j,k})_{ii}({\bf u}) + {v}_i^2\Big( (f_{j,k})_{ii}({\bf u}) - f_{ii}({\bf v})\Big).
\end{eqnarray*}
Thus
\begin{eqnarray*}
\big\|\Delta \Big(f_{j,k}({\bf u}) - f_{j,k}({\bf v})\Big)\big\|_{L_T^{\frac{8p}{p(8-N) + N}}(L^{\frac{2p}{2p-1}})}&\leq&\big\| \displaystyle\sum_i({\bf u} - {\bf v})_{ii}(f_{j,k})_i({\bf u})  \big\|_{L_T^{\frac{8p}{p(8-N) + N}}(L^{\frac{2p}{2p-1}})}\\& +& \big\|  \displaystyle\sum_i {\bf v}_{ii}\Big((f_{j,k})_i({\bf u}) - (f_{j,k})_i({\bf v})\Big)\big\|_{L_T^{\frac{8p}{p(8-N) + N}}(L^{\frac{2p}{2p-1}})} \\&+& \big\|  \displaystyle\sum_i\Big(u_i^2 - v_i^2\Big)(f_{j,k})_{ii}({\bf u})  \big\|_{L_T^{\frac{8p}{p(8-N) + N}}(L^{\frac{2p}{2p-1}})} \\ &+ &\big\|\displaystyle\sum_i |{v}_i|^2\Big( (f_{j,k})_{ii}({\bf u}) - (f_{j,k})_{ii}({\bf v})\Big) \big\|_{L_T^{\frac{8p}{p(8-N) + N}}(L^{\frac{2p}{2p-1}})}\\
&\leq&(\mathcal{I}_1) + (\mathcal{I}_2) +(\mathcal{I}_3) + (\mathcal{I}_4).
\end{eqnarray*}
Via H\"older inequality and Sobolev embedding, we obtain
\begin{eqnarray*}
(\mathcal{I}_1)
&\lesssim&\|\Delta({\bf u} - {\bf v})\|_{L_T^{\frac{8p}{N(p-1)}}(L^{2p})} \big\| |u_k|^{p-1}|u_j|^{p -1}+ {|u_k|^{p}|u_j|^{p-2}}\big\|_{L_T^{\frac{8p}{8p - 2N(p-1)}}(L^{\frac{p}{p-1}})}\\
&\lesssim&T^{\frac{8p - 2N(p-1)}{8p}} \|\Delta({\bf u} - {\bf v})\|_{L_T^{\frac{8p}{N(p-1)}}(L^{2p})}\big\| |u_k|^{p-1}|u_j|^{p-1} + |u_k|^{p}|u_j|^{p-2}  \big\|_{L_T^\infty(L^{\frac{p}{p-1}})} \\
&\lesssim& T^{\frac{8p - 2N(p-1)}{8p}} \|{\bf u} - {\bf v}\|_T\Big(\|u_k\|_{L_T^\infty(L^{2p})}^{p-1}\|u_j\|_{L_T^\infty(L^{2p})}^{p-1} 
+ \|u_k\|_{L_T^\infty(L^{2p})}^p\|u_j\|_{L_T^\infty(L^{2p})}^{p-2} \Big)\\
&\lesssim& T^{\frac{8p - 2N(p-1)}{8p}} \|{\bf u} - {\bf v}\|_T\Big(\|u_k\|_{L_T^\infty(H^2)}^{p-1}\|u_j\|_{L_T^\infty(H^2)}^{p-1} 
+ \|u_k\|_{L_T^\infty(H^2)}^p\|u_j\|_{L_T^\infty(H^2)}^{p-2} \Big).
\end{eqnarray*}
With the same way, 
{\small \begin{eqnarray*}
(\mathcal{I}_2)
&\lesssim& \| \Delta {\bf v}\|_{L_T^{\frac{8p}{N(p-1)}}(L^{2p})} \| {\bf u} - {\bf v}\|_{L^\infty_T(L^{2p})}\big\||u_k|^{p-2}|u_j|^{p-1}+|u_k|^{p}|u_j|^{p-3}\big\|_{L_T^{\frac{8p}{8p - 2N(p-1)}}(L^{\frac{2p}{2p-3}})}\\
&\lesssim&T^{\frac{8p - 2N(p-1)}{8p}}\| \Delta {\bf v}\|_{L_T^{\frac{8p}{N(p-1)}}(L^{2p})} \| {\bf u} - {\bf v}\|_{L^\infty_T(L^{2p})}\big\||u_k|^{p-2}|u_j|^{p-1} + |u_k|^{p}|u_j|^{p-3}\big\|_{L_T^\infty(L^{\frac{2p}{2p-3}})}\\
&\lesssim&T^{\frac{8p - 2N(p-1)}{8p}}\| \Delta {\bf v}\|_{L_T^{\frac{8p}{N(p-1)}}(L^{2p})} \| {\bf u} - {\bf v}\|_{L^\infty_T(L^{2p})}\Big(\|u_k\|_{L_T^\infty(L^{2p})}^{p-2}\|u_j\|_{L_T^\infty(L^{2p})}^{p-1} 
+ \|u_k\|_{L_T^\infty(L^{2p})}^p\|u_j\|_{L_T^\infty(L^{2p})}^{p-3} \Big)\\
&\lesssim&T^{\frac{8p - 2N(p-1)}{8p}}\| \Delta {\bf v}\|_{L_T^{\frac{8p}{N(p-1)}}(L^{2p})} \| {\bf u} - {\bf v}\|_{L^\infty({H^2})}\Big(\|u_k\|_{L_T^\infty({H^2})}^{p-2}\|u_j\|_{L_T^\infty({H^2})}^{p-1} 
+ \|u_k\|_{L_T^\infty({H^2})}^p\|u_j\|_{L_T^\infty({H^2})}^{p-3} \Big).
\end{eqnarray*}}
Arguing as previously,
{\small\begin{eqnarray*}
(\mathcal{I}_3)
&\lesssim&\big\|\displaystyle\sum_i|{u}_i - { v}_i|\Big( |{u}_i| + |{v}_i|\Big)(f_{j,k})_{ii}({\bf u})  \big\|_{L_T^{\frac{8p}{p(8-N) + N}}(L^{\frac{2p}{2p-1}})}\\
&\lesssim&\sum_i\|{u_i} - {v_i}\|_{L_T^{\frac{8p}{N(p-1)}}(L^{2p})}\| |{u}_i| +|{v}_i|\|_{L_T^\infty(L^{2p})}T^{\frac{N(p-1)}{8p}}\big\||u_k|^{p-2}|u_j|^{p-1} + |u_k|^p|u_j|^{p-3}\big\|_{L_T^{\frac{8p}{8p- 3N(p-1)}}(L^{\frac{2p}{2p -3}})}\\
&\lesssim&\sum_i\|{ u_i} - { v_i}\|_{L_T^{\frac{8p}{N(p-1)}}(L^{2p})}\| |{u}_i| +|{v}_i|\|_{L_T^{\infty}(L^{2p})}T^{\frac{N(p-1)}{8p}}\Big(\|u_k^{p-2}\|_{L_T^\infty(L^{\frac{2p}{p -2}})}\|u_j^{p-1}\|_{L_T^\infty(L^{\frac{2p}{p -1}})} \\
&+& \|u_k^p\|_{L_T^\infty(L^2)}\|u_j^{p-3}\|_{L_T^\infty(L^{\frac{4p}{2p-6}})}\Big)T^{\frac{8p- 3N(p-1)}{8p}}\\
&\lesssim&\|{\bf u} - {\bf v}\|_{L_T^{\frac{8p}{N(p-1)}}(L^{2p})}\|{\bf v}\|_{L_T^\infty(H^2)}\Big(\|u_k\|_{L_T^\infty(H^2)}^{p-2}\|u_j\|_{L_T^\infty(H^2)}^{p-1}+ \|u_k\|_{L_T^\infty(H^2)}^p\|u_j\|_{L_T^\infty(H^2)}^{p-3}\Big)T^{\frac{8p- 2N(p-1)}{8p}}.
\end{eqnarray*}}
With the same way
{\small\begin{eqnarray*}
(\mathcal{I}_4)
&\lesssim& \|{\bf u} - {\bf v}\|_{L_T^{\frac{8p}{N(p-1)}}(L^{2p})}\|{\bf v}\|_{L_T^{\infty}(L^{2p})}^2T^{\frac{N(p-1)}{4p}} \big\||u_k|^{p-3}|u_j|^{p-1}+ |u_k|^p |u_j|^{p-4}\big\|_{L_T^{\frac{2p}{2p - N(p-1)}}(L^{\frac{p}{p-2}})} \\
&\lesssim& T^{\frac{2p - N(p-1)}{2p}}\|{\bf u} - {\bf v}\|_{L_T^{\frac{8p}{N(p-1)}}(L^{2p})}\|{\bf v}\|_{L_T^{\infty}(L^{2p})}^2T^{\frac{N(p-1)}{4p}} \big\||u_k|^{p-3}|u_j|^{p-1}+ |u_k|^p |u_j|^{p-4}\big\|_{L_T^\infty(L^{\frac{p}{p-2}})} \\
&\lesssim& T^{\frac{4p - N(p-1)}{4p}}\|{\bf u} - {\bf v}\|_{L_T^{\frac{8p}{N(p-1)}}(L^{2p})}\|{\bf v}\|_{L_T^\infty(H^2)}^2\Big(\|u_k\|_{L_T^\infty(L^{2p})}^{p-3} \|u_j\|_{L_T^\infty(L^{2p})}^{p-1} \\
&+& \|u_k\|_{L_T^\infty(L^{2p})}^{p}\|u_j\|_{L_T^\infty(L^{2p})}^{p-4}  \Big) \\
&\lesssim& T^{\frac{4p - N(p-1)}{4p}}\|{\bf u} - {\bf v}\|_{L_T^{\frac{8p}{N(p-1)}}(L^{2p})}\|{\bf v}\|_{L_T^{\infty}(H^2)}^2\Big(\|u_k\|_{L_T^\infty(H^2)}^{p-3} \|u_j\|_{L_T^\infty(H^2)}^{p-1}+\|u_k\|_{L_T^\infty(H^2)}^{p}\|u_j\|_{L_T^\infty(H^2)}^{p-4} \Big).
\end{eqnarray*}}
Thus, for $T>0$ small enough, $\phi$ is a contraction satisfying
$$\|\phi({\bf u}) - \phi({\bf v})\|_T\lesssim T^{\frac{4p - N(p-1)}{4p}}R^{2p-2}\|{\bf u} - {\bf v}\|_T .$$ 
Taking in the last inequality ${\bf v}=0,$ yields
\begin{eqnarray*}
\|\phi({\bf u})\|_T
&\lesssim& T^{\frac{4p - N(p-1)}{4p}}R^{2p-1}+ \|\phi(0)\|_T\\
&\lesssim& T^{\frac{4p - N(p-1)}{4p}}R^{2p-1}+ TR .
\end{eqnarray*}
Since $1<p\leq p^*$, $\phi$ is a contraction of $ B_T(R)$ for some $R,T>0$ small enough.
\subsection{Uniqueness}
In what follows, we prove uniqueness of solution to the Cauchy problem \eqref{S}. Let $T>0$ be a positive time, ${\bf u},{\bf v}\in C_T(H)$ two solutions to \eqref{S} and ${\bf w} := {\bf u} - {\bf v}.$ Then
$$i\frac{\partial }{\partial t}w_j +\Delta ^2 w_j = \displaystyle \sum_{k=1}^{m}\big( |u_k|^p|u_j|^{p - 2 }u_j -  |v_k|^p|v_j|^{p - 2 }v_j\big),\quad w_j(0,.)= 0.$$
Applying Strichartz estimate with the admissible pair $(q,r) = (\frac{8p}{N(p-1)}, 2p) $, we have
\begin{eqnarray*}
\|{\bf u} - {\bf v}\|_{(L_T^q(L^r))^{(m)}}\lesssim \displaystyle\sum_{j=1}^{m}\displaystyle\sum_{k=1}^{m}\big\|f_{j,k}({\bf u}) -  f_{j,k}({\bf v})\big\|_{L_T^{q^\prime}(L^{r^\prime})}.
\end{eqnarray*}
Taking $T>0$ small enough, whith a continuity argument, we may assume that 
$$ \max_{j=1,...,m}\|u_j\|_{L_T^\infty(H^2)}\leq 1.$$
Using previous computation with$$ (\mathcal{I}) :=\big\|f_{j,k}({\bf u}) -  f_{j,k}({\bf u})\big\|_{L_T^{q^\prime}(L^{r^\prime})}=  \big\||u_k|^p|u_j|^{p-2}u_j - |v_k|^p|v_j|^{p-2}v_j\big\|_{L_T^{q^\prime}(L^{r^\prime})},$$ 
we have
\begin{eqnarray*}
(\mathcal{I})&\lesssim&\big\|\Big(|u_k|^{p-1}|u_j|^{p-1}  + |u_k|^p|u_j|^{p-2} \Big)|{\bf u} - {\bf v}|\big\|_{L_T^{\frac{8p}{p(8-N) + N}}(L^{\frac{2p}{2p-1}})}\\
&\lesssim&\|{\bf u} - {\bf v}\|_{L_T^{\frac{8p}{p(8-N) + N}}(L^{2p})}\big\| |u_k|^{p-1}|u_j|^{p-1} + |u_k|^p|u_j|^{p-2} \big\|_{L_T^\infty(L^{\frac{p}{p-1}})}\\
&\lesssim&\|{\bf u} - {\bf v}\|_{L_T^{\frac{8p}{p(8-N) + N}}(L^{2p})}\Big(\|u_k\|_{L_T^\infty(L^{2p})}^{p-1} \|u_j\|_{L_T^\infty(L^{2p})}^{p-1} +  \|u_k\|_{L_T^\infty(L^{2p})}^{p}\|u_j\|_{L_T^\infty(L^{2p})}^{p-2}  \Big)\\
&\lesssim& T^{\frac{(4 - N)p + N}{4 p}}\|{\bf u} - {\bf v}\|_{L_T^{\frac{8p}{N(p - 1)}}(L^{2p})}\Big(\|u_k\|_{L_T^\infty(H^2)}^{p-1} \|u_j\|_{L_T^\infty(H^2)}^{p-1}+ \|u_k\|_{L_T^\infty(H^2)}^{p}\|u_j\|_{L_T^\infty(H^2)}^{p-2}   \Big).
\end{eqnarray*}
Then
$$ \|{\bf w}\|_{(L_T^q(L^r))^{(m)}}\lesssim  T^{\frac{(4 - N)p + N}{4 p}}\|{\bf w} \|_{(L_T^q(L^r))^{(m)}}.$$
Uniqueness follows for small time and then for all time with a translation argument.
\subsection{Global existence in the subcritical case}
We prove that the maximal solution of \eqref{S} is global in the defocusing case. The global existence is a consequence of energy conservation and previous calculations. Let ${\bf u} \in C([0, T^*), H)$ be the unique maximal solution of \eqref{S}. We prove that ${\bf u}$ is global. By contradiction, suppose that $T^*<\infty.$ Consider for $0< s <T^*,$ the problem
\begin{equation} (\mathcal{P}_s)\label{P1}
\left\{
\begin{array}{ll}
i\frac{\partial }{\partial t}v_j +\Delta ^2v_j = \displaystyle \sum_{k=1}^{m} |v_k|^p|v_j|^{p - 2 }v_j;\\
v_j(s,.) = u_j(s,.).
\end{array}
\right.
\end{equation}
Using the same arguments used in the local existence, we can prove a real $\tau>0$ and a solution ${\bf v} = (v_1,...,v_m)$ to $(\mathcal{P}_s)$ on $C\big([s, s+\tau], H).$ Using the conservation of energy we see that $\tau$ does not depend on $s.$ Thus, if we let $s$ be close to $T^*$ such that $T^*< s + \tau,$ this fact contradicts the maximality of $T^*.$
\section{Global existence in the critical case}
We establish global existence of a solution to \eqref{S} in the critical case $p=p^*$ for small data as claimed in Theorem \ref{glb}.\\ 
Several norms have to be considered in the analysis of the critical case. Letting $I\subset \R$ a time slab, we define the norms
\begin{eqnarray*}
\|u\|_{M(I)} &:= &\|\Delta u\|_{L^{\frac{2(N + 4)}{N-4}}(I, L^{\frac{2N(N + 4)}{N^2  + 16}})};\\
\|u\|_{W(I)}& :=&\|\nabla u\|_{L^{\frac{2(N + 4)}{N-4}}(I, L^{\frac{2N(N + 4)}{N^2 -2N + 8}})};  \\
\|u\|_{Z(I)}& :=&\| u\|_{L^{\frac{2(N + 4)}{N-4}}(I, L^{\frac{2(N + 4)}{N - 4}})};\\
\|u\|_{N(I)}&: =&\|\nabla u\|_{L^2(I, L^{\frac{2N}{N+2}})}.
\end{eqnarray*}
Let $M(\R)$ be the completion of $C_c^\infty(\R^{N+1})$ with the norm $\|.\|_{M(\R)},$ and $M(I)$ be the set consisting of the restrictions to $I$ of functions in $M(\R).$ We adopt similar definitions for $W$ and $N.$
An important quantity closely related to the mass and the energy, is the functional $\xi$ defined for ${\bf u}\in H $ by
$$\xi({\bf u}) = \displaystyle \sum_{j=1}^m\displaystyle\int_{\R^N}|\Delta u_j|^2\,dx.$$
We give an auxiliary result.
\begin{prop}\label{proposition 1}
Let $4<N\leq6$ and $p= p^*.$ There exists $\delta>0$ such that for any initial data $\Psi \in H$ and any interval $I=[0, T],$ if
$$ \displaystyle\sum_{j=1}^{m}\|e^{it\Delta^2}\psi_{j}\|_{W(I)}< \delta,$$
then there exits a unique solution ${\bf u}\in C(I, H)$ of \eqref{S} which satisfies ${\bf u}\in \big(M(I)\cap L^{\frac{2(N+4)}{N}}(I\times \R^N)\big)^{(m)}.$ Moreover,
\begin{gather*}
\displaystyle\sum_{j=1}^{m}\|u_j\|_{W(I)}\leq 2\delta;\\
\displaystyle\sum_{j=1}^{m}\|u_j\|_{M(I)} + \displaystyle\sum_{j=1}^{m}\|u_j\|_{L^\infty(I, H^2)}\leq C(\|\Psi\|_{H^2}+\delta^{\frac{N+4}{N-4}}).
\end{gather*}
Besides, the solution depends continuously on the initial data in the sense that there exists $\delta_0$ depending on $\delta,$ such that for any $\delta_1\in (0,\delta_0),$ if $\displaystyle\sum_{j=1}^{m}\|\psi_{j} - \varphi_{j}\|_{H^2}\leq \delta_1$ and ${\bf v}$ be the local solution of \eqref{S} with initial data $\varphi:=(\varphi_{1,0},...,\varphi_{m,0}),$ then ${\bf v}$ is defined on $I$ and for any admissible couple $(q,r)$,
$$\|{\bf u} - {\bf v}\|_{(L^q(I, L^r))^{(m)}}\leq C\delta_1.$$
\end{prop}
\begin{proof}
The proposition follows from a contraction mapping argument. For ${\bf u}\in( W(I))^{(m)}$, we let $\phi({\bf u})$ given by
$$\phi({\bf u})(t) := T(t){\Psi} -i \displaystyle\sum_{j,k =1}^{m}\displaystyle\int_0^tT(t-s)\Big(|u_k|^{\frac{N}{N-4}}|u_j|^{\frac{8-N}{N-4}}u_j(s)\Big)\,ds.$$
Define the set
$$ X_{M,\delta} := \{ {\bf u}\in (M(I))^{(m)};\,  \displaystyle\sum_{j=1}^{m}\|u_j\|_{W(I)}\leq 2\delta, \, \displaystyle\sum_{j=1}^{m}\|u_j\|_{L^{\frac{2(N+4)}{N}}(I,L^{\frac{2(N+4)}{N}})}\leq 2M\}$$
where $M := C \|\Psi\|_{(L^2)^{(m)}}$ and $\delta>0$ is sufficiently small. Using Strichartz estimate, we get
\begin{eqnarray*}
\|\phi({\bf u}) - \phi({\bf v})\|_{\big({L^{\frac{2(N+4)}{N}}(I,L^{\frac{2(N+4)}{N}})}\big)^{(m)}}&\lesssim& \displaystyle\sum_{j,k=1}^{m}\big\|f_{j,k}({\bf u}) - f_{j,k}({\bf v})\big\|_{L^{\frac{2(N+4)}{N+8}}(I,L^{\frac{2(N+4)}{N+8}})}.
\end{eqnarray*}
Using H\"older inequality and denoting the quantity $ (\mathcal{J}):= \big\|f_{j,k}({\bf u}) - f_{j,k}({\bf v})\big\|_{L^{\frac{2(N+4)}{N+8}}(I,L^{\frac{2(N+4)}{N+8}})}$, we obtain
\begin{eqnarray*}
(\mathcal{J})
&\lesssim&\big\|\Big(|u_k|^{\frac{4}{N-4}}|u_j|^{\frac{4}{N-4}} + |u_k|^{\frac{N}{N-4}}|u_j|^{\frac{8-N}{N-4}}\Big)|{\bf u} - {\bf v}|\big\|_{L_T^{\frac{2(N+4)}{N+8}}(L^{\frac{2(N+4)}{N+8}})}\\
&\lesssim&\|{\bf u} - {\bf v}\|_{L_T^{\frac{2(N+4)}{N}}(L^{\frac{2(N+4)}{N}})}\Big(\big\||u_k|^{\frac{4}{N-4}}|u_j|^{\frac{4}{N-4}}\big\|_{L_T^{\frac{N+4}{4}}(L^{\frac{N+4}{4}})}+ \big\||u_k|^{\frac{N}{N-4}}|u_j|^{\frac{8-N}{N-4}}\big\|_{L_T^{\frac{N+4}{4}}(L^{\frac{N+4}{4}})}\Big)\\
&\lesssim&\|{\bf u} - {\bf v}\|_{L_T^{\frac{2(N+4)}{N}}(L^{\frac{2(N+4)}{N}})}\Big(\|u_k\|_{L_T^{\frac{2(N+4)}{N- 4}}(L^{\frac{2(N+4)}{N - 4}})}^{\frac{4}{N-4}} \|u_j\|_{L_T^{\frac{2(N+4)}{N - 4}}(L^{\frac{2(N+4)}{N- 4}})}^{\frac{4}{N-4}}\\ &+& \|u_k\| _{L_T^{\frac{2(N+4)}{N - 4}}(L^{\frac{2(N+4)}{N - 4}})} ^{\frac{N}{N-4}}\|u_j\|_{L_T^{\frac{2(N+4)}{N - 4}}(L^{\frac{2(N+4)}{N - 4}})}^{\frac{8-N}{N-4}}\Big).
\end{eqnarray*}
By Proposition \ref{injection}, we have the Sobolev embedding
$$ \| u\|_{L^{\frac{2(N+4)}{N-4}}(I, L^{\frac{2(N + 4)}{N - 4}})}\lesssim \|\nabla u\|_{L^{\frac{2(N + 4)}{N-4}}(I, L^{\frac{2N(N + 4)}{N^2 -2N + 8}})}, $$
hence 
\begin{eqnarray*}
(\mathcal{J})
&\lesssim&\|{\bf u} - {\bf v}\|_{L_T^{\frac{2(N+4)}{N}}(L^{\frac{2(N+4)}{N}})}\Big(\|u_k\|_{W(I)}^{\frac{4}{N-4}} \|u_j\|_{W(I)}^{\frac{4}{N-4}}+\|u_k\| _{W(I)} ^{\frac{N}{N-4}}\|u_j\|_{W(I)}^{\frac{8-N}{N-4}}\Big)\\
&\lesssim& \delta^{\frac{8}{N-4}}\|{\bf u} - {\bf v}\|_{L_T^{\frac{2(N+4)}{N}}(L^{\frac{2(N+4)}{N}})}.
\end{eqnarray*}
Then
$$\|\phi({\bf u}) - \phi({\bf v})\|_{\big({L^{\frac{2(N+4)}{N}}(I,L^{\frac{2(N+4)}{N}})}\big)^{(m)}}\lesssim \delta^{\frac{8}{N-4}} \|{\bf u} - {\bf v}\|_{\big({L^{\frac{2(N+4)}{N}}(I,L^{\frac{2(N+4)}{N}})}\big)^{(m)}}.$$
Moreover, taking in the previous inequality ${\bf v=0}$, we get for small $\delta>0$,
\begin{eqnarray*}
\|\phi({\bf u})\|_{\big({L^{\frac{2(N+4)}{N}}(I,L^{\frac{2(N+4)}{N}})}\big)^{(m)}}
&\lesssim&C\|\Psi\|_{(L^2)^m}+ \delta^{\frac{8}{N-4}} M\\
&\lesssim&(1+ \delta^{\frac{8}{N-4}}) M\\
&\leq&2M.
\end{eqnarray*}
With a classical Picard argument, there exists ${\bf u}\in L^{\frac{2(N+4)}{N}}(I,L^{\frac{2(N+4)}{N}})$ a solution to \eqref{S} satisfying
 $$\|{\bf u}\|_{\big(L^{\frac{2(N+4)}{N}}(I,L^{\frac{2(N+4)}{N}})\big)^{(m)}}\leq 2M.$$
Taking account of Strichartz estimate we get,
\begin{eqnarray*}
\| {\bf u}\|_{(M(I))^{(m)}}
&\lesssim& \|\Delta \Psi\|_{({L^2})^{(m)}} +\displaystyle\sum_{j,k=1}^{m} \| \nabla f_{j,k}({\bf u})\|_{L_T^2(L^{\frac{2N}{N +2}})}.
\end{eqnarray*}
Let $(\mathcal{J}_1):= \| \nabla f_{j,k}({\bf u})\|_{L_T^2(L^{\frac{2N}{N +2}})} $. Using H\"older inequality and Sobolev embedding with, yields
\begin{eqnarray*}
(\mathcal{J}_1)
&\lesssim& \big\| |\nabla {\bf u}| \Big( |u_k|^{\frac{4}{N-4}}|u_j|^{\frac{4}{N-4}} + |u_k|^{\frac{N}{N-4}}|u_j|^{\frac{8-N}{N-4}}\Big)\big\|_{L_T^2(L^{\frac{2N}{N +2}})}\\
&\lesssim&\|\nabla{\bf u}\|_{L_T^{\frac{2(N+4)}{N - 4}}(L^{\frac{2N(N+4)}{N^2 - 2N +8}})}\big\||u_k|^{\frac{4}{N-4}}|u_j|^{\frac{4}{N-4}} + |u_k|^{\frac{N}{N-4}}|u_j|^{\frac{8-N}{N-4}}\big\|_{L_T^{\frac{N+4}{4}}(L^{\frac{N+4}{4}})}\\
&\lesssim&\|\nabla{\bf u}\|_{L_T^{\frac{2(N+4)}{N - 4}}(L^{\frac{2N(N+4)}{N^2 - 2N +8}})}\Big(\|u_k\|_{L_T^{\frac{2(N+4)}{N- 4}}(L^{\frac{2(N+4)}{N - 4}})}^{\frac{4}{N-4}} \|u_j\|_{L_T^{\frac{2(N+4)}{N - 4}}(L^{\frac{2(N+4)}{N- 4}})}^{\frac{4}{N-4}}\\ &+& \|u_k\| _{L_T^{\frac{2(N+4)}{N - 4}}(L^{\frac{2(N+4)}{N - 4}})} ^{\frac{N}{N-4}}\|u_j\|_{L_T^{\frac{2(N+4)}{N - 4}}(L^{\frac{2(N+4)}{N - 4}})}^{\frac{8-N}{N-4}}\Big)\\
&\lesssim&\|{\bf u}\|_{(W(I))^{(m)}}\Big(\|u_k\|_{W(I)}^{\frac{4}{N-4}} \|u_j\|_{W(I)}^{\frac{4}{N-4}}+\|u_k\| _{W(I)} ^{\frac{N}{N-4}}\|u_j\|_{W(I)}^{\frac{8-N}{N-4}}\Big).
\end{eqnarray*}
Then
\begin{eqnarray*}
 \|{\bf u}\|_{(M(I))^{(m)}}&\lesssim& \|\Psi\|_H +\displaystyle\sum_{j,k=1}^m\|{\bf u}\|_{(W(I))^{(m)}}\Big(\|u_k\|_{W(I)}^{\frac{4}{N-4}} \|u_j\|_{W(I)}^{\frac{4}{N-4}}+\|u_k\| _{W(I)} ^{\frac{N}{N-4}}\|u_j\|_{W(I)}^{\frac{8-N}{N-4}}\Big)  \\
&\lesssim& \|\Psi\|_H + \delta^{\frac{N + 4}{N - 4}}.
\end{eqnarray*}
By Proposition \ref{injection}, we have the continuous Sobolev embedding
$$ W^{2, \frac{2N(N + 4)}{N^2 + 16}} \hookrightarrow W^{1, \frac{2N(N + 4)}{N^2 - 2N + 8}}.$$
So, it follows that
\begin{equation}\label{*}\| {\bf u}\|_{(W(I))^{(m)}}\lesssim\| {\bf u}\|_{(M(I))^{(m)}}.\end{equation}
Thanks to Strichartz estimates \ref{S1}, we have
\begin{eqnarray*}
\| {\bf u}\|_{(W(I))^{(m)}}
&\lesssim& \delta +\|\int_0^tT(t-s)f_{j,k}(u)\,ds\|_{(W(I))^{(m)}}\\
&\lesssim& \delta +\|\int_0^tT(t-s)f_{j,k}(u)\,ds\|_{(M(I))^{(m)}}\\
&\lesssim& \delta +\| {\bf u}\|_{(W(I))^{(m)}}^{\frac{N + 4}{N - 4}}
\end{eqnarray*}
so, by Lemma \ref{Bootstrap}, 
$$\| {\bf u}\|_{(W(I))^{(m)}}\leq 2\delta. $$
Taking an admissible couple $(q,r)$, we return now to the lipschitz bound $ (\mathcal{J}_2):=\|{\bf u} - {\bf v}\|_{(L^q(I, L^r))^{(m)}}\leq C\delta_1.$ By H\"older inequality and Strichartz estimate, we have
{\small\begin{eqnarray*} 
(\mathcal{J}_2)&\lesssim&  \|\Psi - \varphi\|_{(L^2)^{(m)}} +  \displaystyle\sum_{j,k=1}^{m}\|f_{j,k}({\bf u}) - f_{j,k}({\bf v})\|_{L^{\frac{2(N+4)}{N+8}}(I,L^{\frac{2(N+4)}{N+8}})}\\
&\lesssim&\|\Psi - \varphi\|_{(L^2)^{(m)}} + \displaystyle\sum_{j,k=1}^{m}\big\|\Big(|u_k|^{\frac{4}{N-4}}|u_j|^{\frac{4}{N-4}} + |u_k|^{\frac{N}{N-4}}|u_j|^{\frac{8-N}{N-4}}\Big)|{\bf u} - {\bf v}|\big\|_{L^{\frac{2(N+4)}{N+8}}(I,L^{\frac{2(N+4)}{N+8}})}\\
&\lesssim&\|\Psi - \varphi\|_{(L^2)^{(m)}}+ \displaystyle\sum_{j,k=1}^{m}\|{\bf u} - {\bf v}\|_{L^{\frac{2(N+4)}{N}}(I,L^{\frac{2(N+4)}{N}})}\Big(\|u_k\|_{L^{\frac{2(N+4)}{N- 4}}(I,L^{\frac{2(N+4)}{N - 4}})}^{\frac{4}{N-4}} \|u_j\|_{L^{\frac{2(N+4)}{N - 4}}(I,L^{\frac{2(N+4)}{N- 4}})}^{\frac{4}{N-4}}\\ &+& \|u_k\| _{L^{\frac{2(N+4)}{N - 4}}(I,L^{\frac{2(N+4)}{N - 4}})} ^{\frac{N}{N-4}}\|u_j\|_{L^{\frac{2(N+4)}{N - 4}}(I,L^{\frac{2(N+4)}{N - 4}})}^{\frac{8-N}{N-4}}\Big)\\
&\lesssim&\|\Psi - \varphi\|_{(L^2)^{(m)}} + \delta^{\frac{8}{N - 4}}\| {\bf u} - {\bf v}\|_{\big({L^{\frac{2(N+4)}{N}}(I,L^{\frac{2(N+4)}{N}})}\big)^{(m)}}.
\end{eqnarray*}}
The proof is ended by taking $\delta$ small enough.
\end{proof}
We are ready to prove Theorem \ref{glb}.
\begin{proof}[{\bf Proof of Theorem \ref{glb}}]Denote the homogeneous Sobolev space ${\bf H}=($\.H$^2)^{(m)}$. Using the previous proposition via \eqref{*}, it suffices to prove that $\|{\bf u}\|_{\bf H}$ remains small on the whole interval of existence of ${\bf u}.$ 
Write with conservation of the energy and Sobolev's inequality
\begin{eqnarray*}
\|{\bf u}\|_{\bf H}^2&\leq& 2E(\Psi) +\frac{N -4}{N}\displaystyle \sum_{j,k=1}^{m}\displaystyle \int_{\R^N} |u_j(x,t)|^{\frac{N}{N - 4}} |u_k(x,t)|^{\frac{N}{N - 4}}\,dx \\
&\leq& C\big(  \xi(\Psi) + \xi(\Psi)^{\frac{N}{N - 4}}\big) + C \big(\displaystyle\sum_{j=1}^{m}\|\Delta u_j\|_2^2\big)^{\frac{N}{N - 4}}\\
&\leq& C\big(  \xi(\Psi) + \xi(\Psi)^{\frac{N}{N - 4}}\big) +C\|{\bf u}\|_{\bf H} ^{\frac{2N}{N - 4}}.
\end{eqnarray*}
So by Lemma \ref{Bootstrap}, 
if $\xi(\Psi)$ is sufficiently small, then ${\bf u}$ stays small in the ${\bf H}$ norm.
\end{proof}
\section{Scattering}
For any time slab $I,$ take the Strichartz space 
$$S(I):=C(I, H^2)\cap{L^{\frac{8p}{N(p -1)}}(I, W^{2, 2p})}$$
endowed the complete norm
$$ \|u\|_{S(I)}:= \|u\|_{L^\infty(I, H^2)} + \|u\|_{L^{\frac{8p}{N(p -1)}}(I, W^{2, 2p})}.$$
The first intermediate result is the following.
\begin{lem}
For any time slab $I,$ we have
$$\| {\bf u}(t) - e^{it\Delta^2}\Psi\|_{(S(I))^{(m)}}\lesssim\|{\bf u}\|_{\big(L^\infty(I, L^{2p})\big)^{(m)}}^{\frac{2pN(p-1)-8p}{N(p-1)}}\|{\bf u}\|_{\big(L^{\frac{8p}{N(p-1)}}(I, W^{2,2p})\big)^{(m)}}^{\frac{8p - N(p-1)}{N(p-1)}}.$$
\end{lem}
\begin{proof}
Using Strichartz estimate, we have
$$\| {\bf u}(t) - e^{it\Delta^2}\Psi\|_{(S(I))^{(m)}}\lesssim \displaystyle\sum_{j,k=1}^m \|f_{j,k}({\bf u})\|_{L^{\frac{8p}{p(8 - N) + N}}(I, W^{2,\frac{2p}{2p -1}})}.$$
Thanks to H\"older inequality, we get
\begin{eqnarray*}
\|f_{j,k}\|_{L^\frac{2p}{2p -1}}&\lesssim&\big\||u_k|^p|u_j|^{p - 1}\big\|_{L^\frac{2p}{2p -1}}
\lesssim\|u_k\|_{L^{2p}}^p\|u_j\|_{L^{2p}}^{p -1}.
\end{eqnarray*}
Letting $\mu =\theta:= \frac{8p - N(p - 1)}{2N(p - 1) }$, we get $p - \theta ={\frac{N(p -1)(2p +1) - 8p}{2N(p -1)}}$ and $ p - 1 -\mu={\frac{N(p -1)(2p -1) - 8p}{2N(p -1)}}$. Moreover,
\begin{eqnarray*}
\|f_{j,k}\|_{L^{\frac{8p}{p(8 - N) + N}}(I, L^{\frac{2p}{2p -1}})}&\lesssim& \big\|\|u_k\|_{L^{2p}}^p\|u_j\|_{L^{2p}}^{p -1} \big\|_{L^{\frac{8p}{p(8 - N) + N}}(I)}\\
&\lesssim&\|u_k\|_{L^\infty(I,L^{2p})}^{p -\theta}\|u_j\|_{L^\infty(I,L^{2p})}^{p -1-\mu}\big\|\|u_k\|_{L^{2p}}^\theta\|u_j\|_{L^{2p}}^{\mu} \big\|_{L^{\frac{8p}{p(8 - N) + N}}(I)}\\
&\lesssim&\|u_k\|_{L^\infty(I, L^{2p})}^{p -\theta}\|u_j\|_{L^\infty(I, L^{2p})}^{p -1-\mu}\|u_k\|_{L^{\frac{8p}{N(p -1)}}(I, L^{2p})}^{\theta}\|u_j\|_{L^{\frac{8p}{N(p -1)}}(I, L^{2p})}^{\mu}.
\end{eqnarray*}
Then,
\begin{eqnarray}
\displaystyle\sum_{j,k=1}^m\|f_{j,k}\|_{L^{\frac{8p}{p(8 - N) + N}}(I, L^{\frac{2p}{2p -1}})}&\lesssim&\displaystyle\sum_{j,k=1}^m\|u_k\|_{L^\infty(I, L^{2p})}^{p -\theta}\|u_j\|_{L^\infty(I, L^{2p})}^{p -1-\mu}\|u_k\|_{L^{\frac{8p}{N(p -1)}}(I, L^{2p})}^{\theta}\|u_j\|_{L^{\frac{8p}{N(p -1)}}(I, L^{2p})}^{\mu}\nonumber\\
&\lesssim&\displaystyle\sum_{k=1}^m\|u_k\|_{L^\infty(I, L^{2p})}^{p -\theta}\|u_k\|_{L^{\frac{8p}{N(p -1)}}(I, L^{2p})}^{\theta}\displaystyle\sum_{j=1}^m\|u_j\|_{L^\infty(I, L^{2p})}^{p -1 - \mu}\|u_j\|_{L^{\frac{8p}{N(p -1)}}(I, L^{2p})}^{\mu}\nonumber\\
&\lesssim&\Big(\displaystyle\sum_{k=1}^m\big(\|u_k\|_{L^\infty(I, L^{2p})}^{p -\theta}\big)^2\Big)^{\frac{1}{2}}\Big(\displaystyle\sum_{k=1}^m\big( \|u_k\|_{L^{\frac{8p}{N(p -1)}}(I, L^{2p})}^{\theta}\big)^2\Big)^{\frac{1}{2}}\nonumber\\
&\times&\Big(\displaystyle\sum_{j=1}^m\big(\|u_j\|_{L^\infty(I, L^{2p})}^{p - 1- \mu}\big)^2\Big)^{\frac{1}{2}}\Big(\displaystyle\sum_{j=1}^m\big( \|u_j\|_{L^{\frac{8p}{N(p -1)}}(I, L^{2p})}^{\mu}\big)^2\Big)^{\frac{1}{2}}\nonumber\\
&\lesssim&\|{\bf u}\|_{\big(L^\infty(I, L^{2p})\big)^{(m)}}^{\frac{2pN(p -1) - 8p}{N(p - 1)}}\|{\bf u}\|_{\big(L^{\frac{8p}{N(p -1)}}(I, L^{2p})\big)^{(m)}}^{\frac{8p - N(p - 1)}{N(p -1)}}\label{sct1}.
\end{eqnarray}
It remains to estimate the quantity $(\mathcal{I}):=\|\Delta (f_{j,k}({\bf u}))\|_{{L^{\frac{8p}{p(8 - N) + N}}(I, L^{\frac{2p}{2p -1}})}}.$ Write
\begin{eqnarray*}
(\mathcal{I})&\lesssim&
\sum_{i=1}^m \|\Delta{\bf u} (f_{j,k})_i({\bf u})\|_{{L^{\frac{8p}{p(8 - N) + N}}(I, L^{\frac{2p}{2p -1}})}} + \||\nabla {\bf u}|^2(f_{j,k})_{ii}({\bf u})\|_{{L^{\frac{8p}{p(8 - N) + N}}(I, L^{\frac{2p}{2p -1}})}}\\
&\lesssim& (\mathcal{I}_1)  + (\mathcal{I}_2) .
\end{eqnarray*}
Using H\"older inequality, we obtain
\begin{eqnarray*}
\|\Delta{\bf u} (f_{j,k})_i({\bf u})\|_{L^{\frac{2p}{2p -1}}}&\lesssim&\big\| \Delta{\bf u}\big(|u_k|^{p-1}|u_j|^{p-1} + |u_k|^p|u_j|^{p-2}\big)\big\|_{L^{\frac{2p}{2p -1}}}\\
&\lesssim&\|\Delta{\bf u}\|_{(L^{2p})^m}\Big(\big\||u_k|^{p-1}|u_j|^{p - 1}\big\|_{L^\frac{p}{p - 1}}   + \big\||u_k|^{p }|u_j|^{p - 2}\big\|_{L^\frac{p}{p - 1}} \Big)\\
&\lesssim&\|\Delta{\bf u}\|_{(L^{2p})^m}\Big(\|u_k\|_{L^{2p}}^{p-1} \|u_j\|_{L^{2p}}^{p-1} + \|u_k\|_{L^{2p}}^{p}\|u_j\|_{L^{2p}}^{p-2}  \Big).
\end{eqnarray*} 
Letting $\theta = \mu =\alpha = \beta =:\frac{4p - N(p -1)}{N(p - 1)},$ we get $p -1 -\theta=\frac{N(p -1)p - 4p}{N(p - 1)}$ and
\begin{eqnarray*}
(\mathcal{I}_1)
 &\lesssim& \Big\|\|\Delta{\bf u}\|_{(L^{2p})^m}\Big(\|u_k\|_{L^{2p}}^{p-1} \|u_j\|_{L^{2p}}^{p-1} + \|u_k\|_{L^{2p}}^{p}\|u_j\|_{L^{2p}}^{p-2}  \Big)\Big\|_{L^{\frac{8p}{p(8 - N) +N}}}\\
&\lesssim&\|\Delta {\bf u}\|_{\big(L^{\frac{8p}{N(p -1)}}(I, L^{2p})\big)^{(m)}}\Big(\big\|\|u_k\|_{L^{2p}}^{p-1} \|u_j\|_{L^{2p}}^{p-1}\big\|_{L^{\frac{8p}{8p - 2N(p - 1)}}} + \big\|\|u_k\|_{L^{2p}}^{p}\|u_j\|_{L^{2p}}^{p-2}\big\|_{L^{\frac{8p}{8p - 2N(p - 1)}}}\Big)\\
&\lesssim&\|\Delta {\bf u}\|_{\big(L^{\frac{8p}{N(p -1)}}(I, L^{2p})\big)^{(m)}}\Big( \|u_k\|_{L^\infty(I, L^{2p})}^{p -1 -\theta}\|u_j\|_{L^\infty(I, L^{2p})}^{p -1- \mu}\big\|\|u_k\|_{L^{2p}}^{\theta} \|u_j\|_{L^{2p}}^{\mu}\big\|_{L^{\frac{8p}{8p - 2N(p - 1)}}} \\
&+&\|u_k\|_{L^\infty(I, L^{2p})}^{p - \alpha}\|u_j\|_{L^\infty(I, L^{2p})}^{p -2-\beta}\big\|\|u_k\|_{L^{2p}}^{\alpha} \|u_j\|_{L^{2p}}^{\beta}\big\|_{L^{\frac{8p}{8p - 2N(p - 1)}}} \Big)\\
&\lesssim& \|\Delta {\bf u}\|_{\big(L^{\frac{8p}{N(p -1)}}(I, L^{2p})\big)^{(m)}}\Big( \|u_k\|_{L^\infty(I, L^{2p})}^{p -1 -\theta}\|u_j\|_{L^\infty(I, L^{2p})}^{p -1- \mu}\|u_k\|_{L^{\frac{8p}{N(p -1)}}(I, L^{2p})}^{\theta} \|u_j\|_{L^{\frac{8p}{N(p -1)}}(I, L^{2p})}^{\mu} \\
&+&\|u_k\|_{L^\infty(I, L^{2p})}^{p - \alpha}\|u_j\|_{L^\infty(I, L^{2p})}^{p -2-\beta}\|u_k\|_{L^{\frac{8p}{N(p -1)}}(I, L^{2p})}^{\alpha} \|u_j\|_{L^{\frac{8p}{N(p -1)}}(I, L^{2p})}^{\beta} \Big).
\end{eqnarray*}
Then, with $\mathcal{A}:= \displaystyle\sum_{i,j,k=1}^m\|\Delta{\bf u} (f_{j,k})_i({\bf u})\|_{{L^{\frac{8p}{p(8 - N) + N}}(I, L^{\frac{2p}{2p -1}})}},$ we have
\begin{eqnarray*}
\mathcal{A}
&\lesssim&\|\Delta {\bf u}\|_{\big(L^{\frac{8p}{N(p -1)}}(I, L^{2p})\big)^{(m)}}\Big(\displaystyle\sum_{k=1}^m \|u_k\|_{L^\infty(I, L^{2p})}^{p -1 -\theta}\|u_k\|_{L^{\frac{8p}{N(p -1)}}(I, L^{2p})}^{\theta}\nonumber\\
&\times&\displaystyle\sum_{j=1}^m \|u_j\|_{L^\infty(I, L^{2p})}^{p -1- \mu}\|u_j\|_{L^{\frac{8p}{N(p -1)}}(I, L^{2p})}^{\mu} +\displaystyle\sum_{k=1}^m\|u_k\|_{L^\infty(I, L^{2p})}^{p - \alpha}\|u_k\|_{L^{\frac{8p}{N(p -1)}}(I, L^{2p})}^{\alpha}\nonumber\\
&\times&\displaystyle\sum_{j=1}^m \|u_j\|_{L^\infty(I, L^{2p})}^{p -2-\beta}\|u_j\|_{L^{\frac{8p}{N(p -1)}}(I, L^{2p})}^{\beta} \Big)\nonumber.
\end{eqnarray*}
This implies that
{\small\begin{eqnarray}
\mathcal{A}
&\lesssim&\|\Delta {\bf u}\|_{\big(L^{\frac{8p}{N(p -1)}}(I, L^{2p})\big)^{(m)}}\Big(\Big(\displaystyle\sum_{k=1}^m \big(\|u_k\|_{L^\infty(I, L^{2p})}^{p -1 -\theta}\big)^2\Big)^{\frac{1}{2}}\Big(\displaystyle\sum_{k=1}^m\big(\|u_k\|_{L^{\frac{8p}{N(p -1)}}(I, L^{2p})}^{\theta}\big)^2\Big)^{\frac{1}{2}}\nonumber\\
&\times&\Big(\sum_{j=1}^m \big(\|u_j\|_{L^\infty(I, L^{2p})}^{p -1- \mu}\big)^2\Big)^{\frac{1}{2}}\Big(\sum_{j=1}^m\big(\|u_j\|_{L^{\frac{8p}{N(p -1)}}(I, L^{2p})}^{\mu}\big)^2\Big)^{\frac{1}{2}}\nonumber\\
 &+&\Big(\displaystyle\sum_{k=1}^m\big(\|u_k\|_{L^\infty(I, L^{2p})}^{p - \alpha}\big)^2\Big)^{\frac{1}{2}}\Big(\displaystyle\sum_{k=1}^m\big(\|u_k\|_{L^{\frac{8p}{N(p -1)}}(I, L^{2p})}^{\alpha}\big)^2\Big)^{\frac{1}{2}}\nonumber\\
&\times&\Big(\displaystyle\sum_{j=1}^m\big(\|u_j\|_{L^\infty(I, L^{2p})}^{p -2-\beta} \big)^2\Big)^{\frac{1}{2}}\Big(\displaystyle\sum_{j=1}^m\big(\|u_j\|_{L^{\frac{8p}{N(p -1)}}(I, L^{2p})}^{\beta}\big)^2\Big)^{\frac{1}{2}} \Big)\nonumber\\
&\lesssim&\|{\bf u}\|_{\big(L^{\frac{8p}{N(p -1)}}(I, W^{2,2p})\big)^{(m)}}\|{\bf u}\|_{\big({L^\infty(I, L^{2p})}\big)^{(m)}}^{\frac{2pN(p-1)-8p}{N(p-1)}}\|{\bf u}\|_{\big({L^{\frac{8p}{N(p -1)}}(I, L^{2p})}\big)^{(m)}}^{\frac{8p-2N(p-1)}{N(p-1)}}.\label{sct2}
\end{eqnarray}}
Similarly, with $ \mathcal{B}:=\displaystyle\sum_{j,k=1}^m\||\nabla {\bf u}|^2(f_{j,k})_{ii}({\bf u})\|_{L^{\frac{8p}{p(8-N)+N}}(I, L^{\frac{2p}{2p - 1}})},$ we obtain
\begin{eqnarray}
\mathcal{B}
&\lesssim&\displaystyle\sum_{j,k=1}^m\big\||\nabla {\bf u}|^2\Big(|u_k|^{p-2}|u_j|^{p-1} + |u_k|^{p}|u_j|^{p-3}\Big)\big\|_{L^{\frac{8p}{p(8-N)+N}}(I, L^{\frac{2p}{2p - 1}})}\nonumber\\
&\lesssim&\|{\bf u}\|_{\big({L^{\frac{8p}{N(p-1)}}}(I, W^{2,2p})\big)^{(m)}}^2
 \|{\bf u}\|_{\big({L^\infty(I, L^{2p})}\big)^{(m)}}^{\frac{2pN(p-1)-8p}{N(p-1)}}\|{\bf u}\|_{\big({L^{\frac{8p}{N(p-1)}}}(I, L^{2p})\big)^{(m)}}^{\frac{8p-3N(p-1)}{N(p-1)}}\label{sct3}.
\end{eqnarray}
Finally, thanks to \eqref{sct1}-\eqref{sct2}-\eqref{sct3}, it follows that
$$ \|{\bf u}(t) - e^{it\Delta^2}\Psi\|_{(S(I))^{(m)}}\lesssim\|{\bf u}\|_{\big(L^\infty(I, L^{2p})\big)^{(m)}}^{\frac{2pN(p-1)-8p}{N(p-1)}}\|{\bf u}\|_{\big(L^{\frac{8p}{N(p-1)}}(I, W^{2,2p})\big)^{(m)}}^{\frac{8p - N(p-1)}{N(p-1)}}.$$
\end{proof}{}
The next auxiliary result is about the decay of solution.
\begin{lem}\label{t1}
For any $2<r<\frac{2N}{N - 4},$ we have
$$\displaystyle\lim_{t\to \infty}\|{\bf u}(t)\|_{(L^r)^{(m)}}= 0.$$
\end{lem}
\begin{proof}
Let $\chi \in C^\infty_0(\R^N)$ be a cut-off function and $\varphi_n:=(\varphi_1^n,...,\varphi_m^n)$ be a sequence in $H$ satisfying $\displaystyle\sup_{n}\|\varphi_n\|_{H}<\infty$ and
$$ \varphi_n\rightharpoonup \varphi := (\varphi_1,...,\varphi_m)\in H.$$
Let ${\bf u}_n:=(u_1^n,...,u_m^n)\; (\mbox{respectively}\; {\bf u}:=(u_1,...,u_m))$ be the solution in $C(\R, H)$ to \eqref{S} with initial data $\varphi_n\, (\mbox{respectively}\; \varphi).$ In what follows, we prove a claim.\\
{\bf Claim.}\\
 For every $\epsilon>0,$ there exist $T_\epsilon>0$ and $ n_\epsilon\in \N$ such that
 \begin{equation}\label{chi} \|\chi({\bf u}_n - {\bf u})\|_{(L_{T_\epsilon}^\infty  (L^2))^{(m)}}<\epsilon \quad \mbox{for all}\; n>n_\epsilon.\end{equation}
In fact, we introduce the functions ${\bf v}_n:= \chi {\bf u}_n$ and ${\bf v} :=\chi {\bf u}.$ We compute, $v_j^n(0)  = \chi \varphi_j^n$ and
\begin{eqnarray*} i\partial_t v_j^n + \Delta^2 v_j^n &=& \Delta^2\chi u_j^n + 2 \nabla \Delta\chi \nabla u_j^n +  \Delta\chi\Delta u_j^n + 2 \nabla \chi \nabla\Delta u_j^n \\
&+&2\big(\nabla\Delta\chi\nabla u_j^n + \nabla\chi\nabla\Delta u_j^n + 2\displaystyle\sum_{i=1}^N \nabla\partial_i \chi\nabla\partial_i u_j^n\big) + \chi\big(\displaystyle\sum_{k=1}^m|u_k^n|^p|u_j^n|^{p - 2}u_j^n\big).\end{eqnarray*}
Similarly, $v_j(0)=\chi\phi_j$ and
\begin{eqnarray*} i\partial_t v_j + \Delta^2 v_j &=& \Delta^2\chi u_j + 2 \nabla \Delta\chi \nabla u_j +  \Delta\chi\Delta u_j + 2 \nabla \chi \nabla\Delta u_j \\
&+&2\big(\nabla\Delta\chi\nabla u_j + \nabla\chi\nabla\Delta u_j + 2\displaystyle\sum_{i=1}^N \nabla\partial_i \chi\nabla\partial_i u_j\big) + \chi\big(\displaystyle\sum_{k=1}^m|u_k|^p|u_j|^{p - 2}u_j\big).\end{eqnarray*}
Denoting ${\bf w}_n:= {\bf v}_n - {\bf v}$ and ${\bf z}_n:= {\bf u}_n - {\bf u},$ we have
\begin{eqnarray*} i\partial_t w_j^n + \Delta^2 w_j^n &=& \Delta^2\chi z_j^n+ 4 \nabla \Delta\chi \nabla z_j^n +  \Delta\chi\Delta z_j^n + 4 \nabla \chi \nabla\Delta z_j^n \\
&+& 4\displaystyle\sum_{i=1}^N \nabla\partial_i \chi\nabla\partial_i z_j^n + \chi\big(\displaystyle\sum_{k=1}^m|u_k^n|^p|u_j^n|^{p - 2}u_j^n - \displaystyle\sum_{k=1}^m|u_k|^p|u_j|^{p - 2}u_j\big).\end{eqnarray*}
By Strichartz estimate, we obtain
\begin{eqnarray*}\|{\bf w}_n\|_{\big(L_T^\infty(L^2) \cap L^{\frac{8p}{N(p-1)}}_T(L^{2p})\big)^{(m)}}&\lesssim& \|\chi(\varphi_n - \varphi)\|_{(L^2)^{(m)}} + \|\Delta^2\chi {\bf z}_n\|_{(L^1_T(L^2))^{(m)}}+ 4 \|\nabla \Delta\chi \nabla {\bf z}_n\|_{(L^1_T(L^2))^{(m)}}\\
 &+& 4 \|\nabla \chi \nabla\Delta {\bf z}_n\|_{(L^1(L^2))^{(m)}} + 4\| \nabla\partial_i \chi\nabla\partial_i {\bf z}_n\|_{(L^1(L^2))^{(m)}} \\
&+&\displaystyle\sum_{j,k=1}^m\big\|\chi\big(|u_k^n|^p|u_j^n|^{p - 2}u_j^n - |u_k|^p|u_j|^{p - 2}u_j\big)\big\|_{L^{\frac{8p}{p(8-N) + N}}_T(L^{\frac{2p}{2p-1}})}.
\end{eqnarray*}
Thanks to the Rellich Theorem, up to subsequence extraction, we have
$$\epsilon:=\|\chi(\varphi_n - \varphi)\|\longrightarrow0\quad\mbox{as}\quad n\longrightarrow\infty.$$
Moreover, by the conservation laws via properties of $\chi$,
\begin{eqnarray*}
 \mathcal{I}_1
&:=&\|\Delta^2\chi {\bf z}_n\|_{(L^1_T(L^2))^{(m)}}+ 4 \|\nabla \Delta\chi \nabla {\bf z}_n\|_{(L^1_T(L^2))^{(m)}}+ 4 \|\nabla \chi \nabla\Delta {\bf z}_n\|_{(L^1_T(L^2))^{(m)}} + 4\| \nabla\partial_i \chi\nabla\partial_i {\bf z}_n\|_{(L^1_T(L^2))^{(m)}}\\
&\lesssim&\| {\bf z}_n\|_{(L^1_T(L^2))^{(m)}}+  \| \nabla {\bf z}_n\|_{(L^1_T(L^2))^{(m)}}+  \| \nabla\Delta{\bf z}_n\|_{(L^1_T(L^2))^{(m)}} + \|\nabla\partial_i {\bf z}_n\|_{(L^1_T(L^2))^{(m)}}\\
&\lesssim& CT,
\end{eqnarray*}
where $$ C:= \|{\bf u}\|_{(L^\infty(\R,H^2))^{(m)}} + \|{\bf u}_n\|_{(L^\infty(\R,H^2))^{(m)}} .$$
Arguing as previously, we have
\begin{eqnarray*}
\mathcal{I}_2&:=&\|\chi(|u_k^n|^p|u_j^n|^{p-2}u_j^n - |u_k|^p|u_j|^{p-2}u_j)\|_{L^{\frac{8p}{p(8 - N) + N}}_T(L^{\frac{2p}{2p -1}})}\\
&\lesssim&\|\chi(|u_k^n|^{p -1}|u_j^n|^{p-1} - |u_k|^p|u_j|^{p-2})|{\bf u}_n - {\bf u}|\|_{L^{\frac{8p}{p(8 - N) + N}}_T(L^{\frac{2p}{2p -1}})}\\
&\lesssim&\|\chi({\bf u}_n  - {\bf u})\|_{L^{\frac{8p}{p(8 - N) + N}}_T((L^{2p})^{(m)})}\Big( \|u_k^n\|_{L^\infty_T(L^{2p})} ^{p-1}\|u_j^n\|_{L^\infty_T(L^{2p})} ^{p-1}  + \|u_k\|_{L^\infty_T(L^{2p})} ^{p}\|u_j\|_{L^\infty_T(L^{2p})} ^{p-2}\Big)\\
&\lesssim&T^{\frac{8p - 2N(p-1)}{8p}}\|{\bf w}_n \|_{L^{\frac{8p}{N(p -1)}}_T((L^{2p})^{(m)})}\Big( \|u_k^n\|_{L^\infty_T(H^2)} ^{p-1}\|u_j^n\|_{L^\infty_T(H^2)} ^{p-1}  + \|u_k\|_{L^\infty_T(H^2)} ^{p}\|u_j\|_{L^\infty_T(H^2)} ^{p-2}\Big)\\
&\lesssim&T^{\frac{8p - 2N(p-1)}{8p}}\|{\bf w}_n \|_{L^{\frac{8p}{N(p -1)}}_T((L^{2p})^{(m)})}\Big( \|u_k^n\|_{L^\infty_T(H^2)} ^{2(p-1)} + \|u_j^n\|_{L^\infty_T(H^2)} ^{2(p-1)}  + \|u_k\|_{L^\infty_T(H^2)} ^{2p} + \|u_j\|_{L^\infty_T(H^2)} ^{2(p-2)}\Big)\\
&\lesssim&T^{\frac{8p - 2N(p-1)}{8p}}\|{\bf w}_n \|_{L^{\frac{8p}{N(p -1)}}_T((L^{2p})^{(m)})}.
\end{eqnarray*}
As a consequence
\begin{eqnarray*}
\|{\bf w}_n\|_{\big(L_T^\infty(L^2) \cap L^{\frac{8p}{N(p-1)}}_T(L^{2p})\big)^{(m)}}
&\lesssim& \epsilon + CT + T^{\frac{8p - 2N(p-1)}{8p}}\|{\bf w}_n \|_{L^{\frac{8p}{N(p -1)}}((L^{2p})^{(m)})}\\
&\lesssim& \frac{\epsilon + T}{1 - T^{\frac{8p - 2N(p-1)}{8p}}}.
\end{eqnarray*}
The claim is proved.\\
By an interpolation argument it is sufficient to prove the decay for $r:= 2 +\frac{4}{N}.$ We recall the following Gagliardo-Nirenberg inequality
\begin{equation}\label{GN}
\|u_j(t)\|_{2 + \frac{4}{N}}^{2 + \frac{4}{N}}\leq C \|u_j(t)\|_{H^2}^2
\Big(\displaystyle\sup_x \|u_j(t)\|_{L^2(Q_1(x))}\Big)^{\frac{4}{N}},
\end{equation}
where $Q_a(x)$ denotes the square centered at $x$ whose edge has length $a$.
We proceed by contradiction. Assume that there exist a sequence $(t_n)$ of positive real numbers and $\epsilon >0$ such that $\displaystyle\lim_{n\to \infty}t_n =\infty$ and 
\begin{equation} \label{IN}\|u_j(t_n)\|_{L^{2 + \frac{4}{N}}}>\epsilon\quad \mbox{for all}\; n\in \N. \end{equation}
By \eqref{GN} and \eqref{IN}, there exist a sequence $(x_n)$ in $\R^N$ and a positive real number denoted also by $\epsilon>0$ such that
\begin{equation}\label{IN1}\|u_j(t_n)\|_{L^2(Q_1(x_n))}\geq\epsilon,\quad \mbox{for all}\; n\in \N. \end{equation}
Let $\phi_j^n(x):=u_j(t_n,x +x_n).$ Using the conservation laws, we obtain
$$ \sup_n\|\phi_j^n\|_{H^2}<\infty.$$
Then, up to a subsequence extraction, there exists $\phi_j\in H^2$ such that $\phi_j^n$ convergence weakly to $\phi_j$ in $H^2.$ By Rellich Theorem, we have 
$$ \displaystyle\lim_{n\to \infty}\|\phi_j^n - \phi_j\|_{L^2(Q_1(0))}=0.$$
Moreover, thanks to \eqref{IN1} we have, $\|\phi_j^n\|_{L^2(Q_1(0))}\geq \epsilon.$ So, we obtain
$$\|\phi_j\|_{L^2(Q_1(0))}\geq \epsilon.$$
We denote by $\bar{u}_j\in C(\R, H^2)$ the solution of \eqref{S} with data $\phi_j$ and ${u}_j^n\in C(\R, H^2)$ the solution of \eqref{S} with data $\phi_j^n.$ Take a cut-off function $\chi \in C_0^\infty(\R^N)$ which satisfies $0\leq \chi\leq1,\; \chi=1$ on $Q_1(0)$ and $supp(\chi)\subset Q_2(0).$ Using a continuity argument, there exists $T>0$ such that
$$\displaystyle\inf_{t\in[0, T]}\|\chi \bar{u}_j(t) \|_{L^2(\R^N)}\geq \frac{\epsilon}{2}.$$
Now, taking account of the claim \eqref{chi}, there is a positive time denoted also $T$ and $n_\epsilon\in \N$ such that $$\|\chi(u_j^n - \bar{u}_j)\|_{L_T^\infty(L^2)}\leq \frac{\epsilon}{4}\quad \mbox{for all}\; n\geq n_\epsilon.$$
Hence, for all $t\in [0, T]$ and $n\geq n_\epsilon,$
$$ \|\chi u_j^n(t)\|_{L^2}\geq \|\chi \bar{u}_j(t)\|_{L^2} - \|\chi(u_j^n - \bar{u}_j)(t)\|_{L^2}\geq \frac{\epsilon}{4}.$$
Using a uniqueness argument, it follows that $u^n_j(t,x)=u_j(t+t_n,x+x_n)$. Moreover, by the properties of $\chi$ and the last inequality, for all $t\in[0, T]$ and $n\geq n_\epsilon,$
$$ \|u_j(t+t_n)\|_{L^2(Q_2(x_n))}\geq \frac{\epsilon}{4}.$$
This implies that 
$$\|u_j(t)\|_{L^2(Q_2(x_n))}\geq \frac{\epsilon}{4},\quad \mbox{for all}\; t\in [t_n, t_n + T]\;\mbox{and all}\; n\geq n_\epsilon.$$
Moreover, as $\displaystyle\lim_{n\to \infty}t_n=\infty,$
 we can suppose that $t_{n +1}- t_n>T$ for $n\geq n_\epsilon.$ Therefore, thanks to Morawetz estimates \eqref{mrwtz1}, we get for $N>5,$ the contradiction
\begin{eqnarray*}
1 &\gtrsim&\displaystyle\int_0^\infty\displaystyle\int_{\R^N\times\R^N}\frac{|u_j(t,x)|^2|u_j(t,y)|^2}{|x - y|^5}\,dxdydt\\
 &\gtrsim&\displaystyle\sum_n\displaystyle\int_{t_n}^{t_{n} +T}\displaystyle\int_{Q_2(x_n)\times Q_2(x_n)}|u_j(t,x)|^2|u_j(t,y)|^2\,dxdydt\\
&\gtrsim& \displaystyle\sum_nT\big(\frac{\epsilon}{4}\big)^4 = \infty.
\end{eqnarray*}
Using \eqref{mrwtz2}, for $N=5$, write
\begin{eqnarray*}
1
&\gtrsim&\int_0^{\infty}\|u_j(t)\|_{L^4(\R^5)}^4dt\\
&\gtrsim&\sum_n\int_{t_n}^{t_n+T}\|u_j(t)\|_{L^4(Q_2(x_n))}^4dt\\
&\gtrsim&\sum_n\int_{t_n}^{t_n+T}\|u_j(t)\|_{L^2(Q_2(x_n))}^4dt\\
&\gtrsim&\sum_n(\frac\varepsilon4)^4T=\infty.
\end{eqnarray*}
This completes the proof of Lemma \ref{t1}.\\
Finally, we are ready to prove scattering. By the two previous lemmas we have 
$$ \|{\bf u}\|_{(S(t,\infty))^{(m)}}\lesssim \|\Psi\|_{H} + \epsilon(t) \|{\bf u}\|_{(S(t,\infty))^{(m)}}^{\frac{8p - N(p-1)}{N(p-1)}},$$
where $ \epsilon(t)\to 0, \; \mbox{as}\; t\to \infty.$ It follows from Lemma \ref{Bootstrap} that 
$${\bf u} \in (S(\R))^{(m)}.$$
Now, let ${\bf v}(t)= e^{-it\Delta^2}{\bf u}(t).$ Taking account of Duhamel formula 
$${\bf v}(t)= \Psi + i\displaystyle\sum_{j,k=1}^m\displaystyle\int_0^t e^{-is\Delta^2}\big(|u_k|^p|u_j|^{p-2}u_j(s) \big)\, ds.$$
Thanks to \eqref{sct1},\eqref{sct2} and \eqref{sct3}, 
$$f_{j,k}({\bf u})\in L^{\frac{8p}{p(8-N)}}(\R, W^{2, \frac{2p}{2p -1}}),$$ 
so, applying Strichartz estimate, we get for $0<t<\tau,$
\begin{eqnarray*}
\|{\bf v}(t) - {\bf v}(\tau)\|_{H}
&\lesssim&\displaystyle\sum_{j,k=1}^m\big\||u_k|^p|u_j|^{p-2}u_j \big\|_{L^{\frac{8p}{p(8-N)}}((t,\tau), W^{2, \frac{2p}{2p -1}})}\stackrel{t,\tau\rightarrow\infty}{\longrightarrow0}.
\end{eqnarray*}
Taking $u_\pm:=\lim_{t\rightarrow\pm\infty}{\bf v}(t)$, we get 
$$\lim_{t\rightarrow\pm\infty}\|{\bf u}(t)-e^{it\Delta^2}u_{\pm}\|_{H^2}=0.$$
Scattering is proved.
\end{proof}{}
\section{Appendix}
\subsection{Blow-up criterion}
We give a useful criterion for global existence in the critical case.
\begin{prop}Let $p= \frac{N}{N-4}$ and ${\bf u} \in C([0, T), H)$ be a solution of \eqref{S} satisfying $\| {\bf u}\|_{(Z([0, T]))^{(m)}}<+\infty.$ Then, there exists $K:= K ( \| \Psi\|_{H},\, \| {\bf u}\|_{(Z([0, T]))^{(m)}}),$ such that
\begin{equation}\label{G}
\| {\bf u}\|_{\big({L^{\frac{2(N+4)}{N}}([0, T], L^{\frac{2(N+4)}{N}})}\big)^{(m)}} + \| {\bf u}\|_{\big({L^\infty([0, T], {\bf H})}\big)^{(m)}}+ \| {\bf u}\|_{(M([0, T]))^{(m)}}\leq K
\end{equation}
and ${\bf u}$ can be extended to a solution $ \tilde{{\bf u}} \in C([0, T^\prime), H)$ of \eqref{S} for some $T^\prime > T.$
\end{prop}
\begin{proof}
Let $\eta>0$ a small real number and $ M:=\| {\bf u}\|_{(Z([0, T]))^{(m)}}$. The first step is to establish \eqref{G}. In order to do so, we subdivide $[0, T]$ into $n$ slabs $I_j$ such that 
 $$n \sim (1 + \frac{M}{\eta})^{\frac{2(N+4)}{N - 4}}\quad\mbox{and}\quad \| {\bf u}\|_{(Z([0, T]))^{(m)}} \leq \eta. $$
Denote $ (\mathcal{A}):=\| {\bf u}\|_{(M([t_j, t]))^{(m)}}$ and $ I_j = [t_j, t_{j+1}]$. For $t\in I_j,$ by Strichartz estimate and arguing as previously 
\begin{eqnarray*}
(\mathcal{A})- \|{\bf u}(t_j)\|_{\bf H} 
&\lesssim& \| \nabla f_{j,k}({\bf u})\|_{\big(L^2([t_j, t],L^{\frac{2N}{N +2}})\big)^{(m)}}\\
&\lesssim&\displaystyle\sum_{j,k=1}^{m}\|\nabla{\bf u}\|_{L^{\frac{2(N+4)}{N - 4}}([t_j, t],L^{\frac{2N(N+4)}{N^2 - 2N +8}})}\Big(\|u_k\|_{L^{\frac{2(N+4)}{N- 4}}([t_j, t],L^{\frac{2(N+4)}{N - 4}})}^{\frac{4}{N-4}} \|u_j\|_{L^{\frac{2(N+4)}{N - 4}}([t_j, t],L^{\frac{2(N+4)}{N- 4}})}^{\frac{4}{N-4}}\\ &+& \|u_k\| _{L^{\frac{2(N+4)}{N - 4}}([t_j, t],L^{\frac{2(N+4)}{N - 4}})} ^{\frac{N}{N-4}}\|u_j\|_{L^{\frac{2(N+4)}{N - 4}}[t_j, t],(L^{\frac{2(N+4)}{N - 4}})}^{\frac{8-N}{N-4}}\Big)\\
&\lesssim& \|{\bf u}\|_{(W([t_j, t]))^{(m)}}\|{\bf u}\|_{(Z([t_j, t]))^{(m)}}^{\frac{8}{N-4}}\\
&\lesssim& \|{\bf u}\|_{(M([t_j, t]))^{(m)}}\|{\bf u}\|_{(Z([t_j, t]))^{(m)}}^{\frac{8}{N-4}}\lesssim \eta^{\frac{8}{N-4}}\|{\bf u}\|_{(M([t_j, t]))^{(m)}}.
\end{eqnarray*}
Take $( \mathcal{B}):=\|{\bf u}\|_{\big({L^{\frac{2(N+4)}{N}}([t_j, t],L^{\frac{2(N+4)}{N}})}\big)^{(m)}}$. Applying Strichartz estimates, we get
\begin{eqnarray*}
(\mathcal{B})-C\|{\bf u}(t_j)\|_{(L^2)^{(m)}} 
&\leq& C\displaystyle\sum_{j,k=1}^{m}\||u_k|^{\frac{N}{N - 4}}|u_j|^{\frac{8 - N}{N - 4}}u_j\|_{L^{\frac{2(N + 4)}{N + 8}}([t_j, t], L^{\frac{2(N + 4)}{N + 8}})}\\
&\leq& C\displaystyle\sum_{j,k=1}^{m}\displaystyle\big\||u_k|^{\frac{N}{N - 4}}|u_j|^{\frac{8 - N}{N - 4}}\big\| _{L^{\frac{N + 4}{N }}([t_j, t], L^{\frac{N + 4}{N }})}\|u_j\|_{L^{\frac{2(N + 4}{N }}([t_j, t], L^{\frac{2(N + 4)}{N}})}\\
&\leq& C\displaystyle\sum_{j,k=1}^{m}\|u_k\|  _{L^{\frac{2(N + 4)}{N - 4 }}([t_j, t], L^{\frac{2(N + 4)}{N - 4}})}^{\frac{N}{N - 4}} \|u_j\| _{L^{\frac{2(N + 4)}{N - 4 }}([t_j, t], L^{\frac{2(N + 4)}{N - 4}})}^{\frac{8 - N}{N - 4}}\|u_j\|_{L^{\frac{2(N + 4)}{N }}([t_j, t], L^{\frac{2(N + 4)}{N}})}\\
&\leq&  C \|{\bf u}\| _{\big(L^{\frac{2(N + 4)}{N - 4 }}([t_j, t], L^{\frac{2(N + 4)}{N - 4}})\big)^{(m)}}^{\frac{8 }{N - 4}}\|{\bf u}\|_{\big(L^{\frac{2(N + 4)}{N }}([t_j, t], L^{\frac{2(N + 4)}{N}})\big)^{(m)}}\\
&\leq& C \|{\bf u}\| _{(Z([t_j, t]))^{(m)}}^{\frac{8 }{N - 4}}\|{\bf u}\|_{\big(L^{\frac{2(N + 4)}{N }}([t_j, t], L^{\frac{2(N + 4)}{N}})\big)^{(m)}}\\
&\leq& C \eta^{\frac{8 }{N - 4}}\|{\bf u}\|_{\big(L^{\frac{2(N + 4)}{N }}([t_j, t], L^{\frac{2(N + 4)}{N}})\big)^{(m)}}.
\end{eqnarray*}
If $\eta$ is sufficiently small, with conservation of the mass, yields
$$ \|{\bf u}\|_{\big({L^{\frac{2(N+4)}{N}}([t_j, t],L^{\frac{2(N+4)}{N}})}\big)^{(m)}} \leq C\|\Psi\|_{(L^2)^{(m)}}$$
and
$$ \| {\bf u}\|_{(M([t_j, t]))^{(m)}} \leq C \|{\bf u}(t_j)\|_{\bf H}. $$
Applying again Strichartz estimates, yields
$$ \| {\bf u}\|_{\big( L^\infty([t_j, t], {\bf H})\big)^{(m)}} \leq C \|{\bf u}(t_j)\|_{\bf H} . $$
In particular, $\|{\bf u}(t_{j +1})\|_{\bf H} \leq C \|{\bf u}(t_j)\|_{\bf H}$. Finally,
$$ \| {\bf u}\|_{\big( L^\infty([t_j, t], {\bf H})\big)^{(m)}}+\| {\bf u}\|_{(M([t_j, t]))^{(m)}} \leq2 C^n\|\Psi\|_{\bf H}<+\infty.$$
The first step is done. Choose $t_0\in I_n,$ Duhamel's formula gives 
\begin{eqnarray*}
{\bf u}(t) = e^{i(t - t_0)\Delta^2}{\bf u}(t_0) - i \displaystyle\sum_{j,k=1}^{m}\int_{t_0}^te^{i(t - s)\Delta^2}\Big(|u_k|^{\frac{N}{N - 4}}|u_j|^{\frac{8 - N}{N - 4}}u_j(s)\Big)\,ds.
\end{eqnarray*}
Thanks to Sobolev inequality and Strichartz estimate,
\begin{eqnarray*}
\|e^{i(t - t_0)\Delta^2}{\bf u}(t_0)\|_{(W([t_0, t]))^{m}}
&\leq &\|{\bf u}\|_{(W([t_0, t]))^{m}} +C \displaystyle\sum_{j,k=1}^{m}\big\||u_k|^{\frac{N}{N - 4}}|u_j|^{\frac{8 - N}{N - 4}}u_j \big\|_{N([t_0, t])}\\
&\leq &\|{\bf u}\|_{(W([t_0, t]))^{m}} +C\|{\bf u}\|_{(W([t_0, t]))^{m}}^{\frac{N + 4}{N - 4}}.
\end{eqnarray*}
Dominated convergence ensures that the $ \|{\bf u}\|_{(W([t_0, T]))^{m}}$ can be made arbitrarily small as $t_0\to T,$ then
$$\|e^{i(t - t_0)\Delta^2}{\bf u}(t_0)\|_{(W([t_0, T]))^{m}}\leq \delta,$$
where $\delta$ is as in Proposition \ref{proposition 1}. In particular, we can find $t_1\in (0, T)$ and $T^{\prime }>T$ such that
$$\|e^{i(t - t_0)\Delta^2}{\bf u}(t_0)\|_{(W([t_1, T^{\prime}]))^{m}}\leq \delta.$$
Now, it follows from Proposition \ref{proposition 1} that there exists ${\bf v}\in C([t_1, T'], H)$ such that ${\bf v}$ solves \eqref{S} with $ p = \frac{N}{N -4}$ and ${\bf u}(t_1) = {\bf v}(t_1).$ By uniqueness, ${\bf u} = {\bf v}$ in $[t_1, T)$ and ${\bf u}$ can be extended in $[0, T'].$
\end{proof}
\subsection{ Morawetz estimate}
In what follows we give a classical proof, inspired by \cite{Colliander, Miao 2}, of Morawetz estimates. Let ${\bf u}:=(u_1,...,u_m)\in H$ be solution to 
$$i\partial_t u_j +\Delta^2 u_j+ \displaystyle\sum_{k=1}^{m}a_{jk}|u_k|^p|u_j|^{p-2}u_j=0 $$
in $N_1$-spatial dimensions and ${\bf v}:=(v_1,...,v_m)\in H$ be solution to
$$i\partial_t v_j +\Delta^2 v_j+ \displaystyle\sum_{k=1}^{m}a_{jk}|v_k|^p|v_j|^{p-2}v_j =0$$
in $N_2$-spatial dimensions. Define the tensor product ${\bf w}:= ({\bf u}\otimes{\bf v})(t,z)$ for $z$ in
$$  \R^{N_1 +N_2}:= \{ (x,y)\quad\mbox{s. t}\quad x\in \R^{N_1}, y\in \R^{N_2}\}$$
by the formula
$$ ({\bf u}\otimes{\bf v})(t,z) = {\bf u}(t,x){\bf v}(t,y) .$$
Denote $F({\bf u}):= \displaystyle\sum_{k=1}^{m}a_{jk}|u_k|^p|u_j|^{p-2}u_j.$ A direct computation shows that ${\bf w}:=(w_1,...,w_n)= {\bf u}\otimes{\bf v}$ solves the equation
\begin{equation}\label{tensor1}
i\partial_t w_j +\Delta^2 w_j+ F({\bf u})\otimes v_j +  F({\bf v})\otimes u_j:=i\partial_t w_j +\Delta^2 w_j+ h=0
\end{equation}
where $\Delta^2:= \Delta_x^2 + \Delta_y^2.$ Define the Morawetz action corresponding to ${\bf w}$ by
\begin{eqnarray*}
M_a^{\otimes_2}
&:=& 2\displaystyle\sum_{j=1}^m\displaystyle\int_{\R^{N_1}\times \R^{N_2}}\nabla a(z).\Im(\overline{u_j\otimes v_j(z)}\nabla (u_j\otimes v_j)(z))\,dz\\
&=& 2\displaystyle\int_{\R^{N_1}\times \R^{N_2}}\nabla a(z).\Im({\bf \bar{w}}(z)\nabla ({\bf\bar w})(z))\,dz,
\end{eqnarray*}
where $\nabla: =(\nabla_x,\nabla_y).$ It follows from equation \eqref{tensor1} that
\begin{gather*}
 \Im(\partial_t \bar{w}_j\partial_i w_j) =\Re (-i\partial_t \bar{w}_j\partial_i w_j)= - \Re \big((\Delta^2 \bar{w}_j +\displaystyle\sum_{k=1}^{m}a_{jk}|\bar{u}_k|^p|\bar{u}_j|^{p-2}\bar{u}_j \bar{v}_j +\displaystyle\sum_{k=1}^{m}a_{jk}|\bar{v}_k|^p|\bar{v}_j|^{p-2}\bar{v}_j \bar{u}_j)\partial_i w_j\big);\\
 \Im( \bar{w}_j\partial_i\partial_t w_j) =\Re (-i \bar{w}_j\partial_i\partial_t w_j)=\Re \big(\partial_i(\Delta^2 w_j +\displaystyle\sum_{k=1}^{m}a_{jk}|u_k|^p|u_j|^{p-2}u_j v_j +\displaystyle\sum_{k=1}^{m}a_{jk}|v_k|^p|v_j|^{p-2}v_j u_j) \bar{w}_j\big). 
\end{gather*}
Moreover, denoting the quantity $ \big\{ h,w_j\big\}_p:=\Re \big(h\nabla\bar{w}_j - w_j\nabla\bar{h} \big)$, we compute
\begin{eqnarray*}
 \big\{ h,w_j\big\}_p^i
& = &\partial_i\Big(\displaystyle\sum_{k=1}^{m}a_{jk}|\bar{u}_k|^p|\bar{u}_j|^{p-2}\bar{u}_j \bar{v}_j +\displaystyle\sum_{k=1}^{m}a_{jk}|\bar{v}_k|^p|\bar{v}_j|^{p-2}\bar{v}_j \bar{u}_j\Big) w_j\\
& -& \Big(\displaystyle\sum_{k=1}^{m}a_{jk}|u_k|^p|u_j|^{p-2}u_j v_j +\displaystyle\sum_{k=1}^{m}a_{jk}|v_k|^p|v_j|^{p-2}v_j u_j\Big) \partial_i\bar{w}_j.
\end{eqnarray*}
It follows that
\begin{eqnarray*}
\partial_t M_a^{\otimes_2}&=& 2\displaystyle\sum_{j=1}^m\displaystyle\int_{\R^{N_1}\times \R^{N_2}}\partial_i a\Re \big(\bar{w}_j\partial_i \Delta^2 w_j - \partial_iw_j \Delta^2\bar{w}_j\big)\,dz  - 2\displaystyle\sum_{j=1}^m \displaystyle\int_{\R^{N_1}\times \R^{N_2}}\partial_i a \big\{h,w_j\big\}_p^i\,dz\\
&=&-2\displaystyle\sum_{j=1}^m\displaystyle\int_{\R^{N_1}\times \R^{N_2}}\big[\Delta a\Re(\bar{w}_j \Delta^2 w_j) +2\Re( \partial_i a\partial_i\bar{w}_j \Delta^2 w_j)\big] \,dz - \displaystyle\sum_{j=1}^m2 \displaystyle\int_{\R^{N_1}\times \R^{N_2}}\partial_i a \big\{h,w_j\big\}_p^i\,dz\\
&:=& \mathcal{I}_1+ \mathcal{I}_2 - 2\displaystyle\sum_{j=1}^m \displaystyle\int_{\R^{N_1}\times \R^{N_2}}\partial_i a \big\{h,w_j\big\}_p^i\,dz.
\end{eqnarray*}
Similar computations done in \cite{Miao 2}, give
\begin{eqnarray*}
\mathcal{I}_1 + \mathcal{I}_2 &=& 2 \displaystyle\sum_{j=1}^m \Re \displaystyle\int_{\R^{N_1}\times \R^{N_2}}\Big\{2\big(\partial_{ik}^x\Delta_x a \partial_i\bar{u}_j\partial_k u_j|v_j|^2 + \partial_{ik}^y\Delta_y a \partial_i\bar{v}_j\partial_k v_j|u_j|^2 \big) - \frac{1}{2}(\Delta_x^3 + \Delta_y^3) a |u_jv_j|^2 \\
&+& \big(\Delta_x^2 a |\nabla u_j|^2|v_j|^2  + \Delta_y^2 a |\nabla v_j|^2|u_j|^2\big) - 4\big(\partial_{ik}^x a \partial_{i_1 i}\bar{u}_j\partial_{i_1k}u_j|v_j|^2 +\partial_{ik}^y a \partial_{i_1 i}\bar{v}_j\partial_{i_1k}v_j|u_j|^2 \big)\Big\}\,dz.
\end{eqnarray*}
Now we take $a(z):=a(x,y) = |x - y|$ where $(x,y)\in\R^{N}\times \R^{N}.$ Then calculation done in \cite{Miao 2}, yield 
$$\partial_tM_a^{\otimes_2}\leq2 \displaystyle\sum_{j=1}^m \Re \displaystyle\int_{\R^{N_1}\times \R^{N_2}}\Big(- \frac{1}{2}(\Delta_x^3 + \Delta_y^3)a|u_jv_j|^2 - 2\partial_i a\{h,w_j\}_p^i \Big)\,dz.$$
Hence, we get  
$$ \displaystyle\sum_{j=1}^m \displaystyle\int_0^T \displaystyle\int_{\R^{N_1}\times \R^{N_2}}\Big((\Delta_x^3 + \Delta_y^3)a|u_jv_j|^2 +4\partial_i a\{h,w_j\}_p^i \Big)\,dz\,dt \leq\displaystyle\sup_{[0,T]}|M_a^{\otimes_2}|.$$
Then
\begin{eqnarray*} & &\displaystyle\sum_{j=1}^m \displaystyle\int_0^T \displaystyle\int_{\R^{N_1}\times \R^{N_2}}\Big((\Delta_x^3 + \Delta_y^3)a|u_jv_j|^2 +4(1 - \frac{1}{p})\Delta_x a\displaystyle\sum_{k=1}^ma_{jk}|u_k|^p|u_j|^p|v_j|^2 \\
&+& 4(1 - \frac{1}{p})\Delta_y a\displaystyle\sum_{k=1}^ma_{jk}|v_k|^p|v_j|^p|u_j|^2 \Big)\,dz\,dt \leq\displaystyle\sup_{[0,T]}|M_a^{\otimes_2}|.\end{eqnarray*}
Taking account of the equalities $\Delta_x a = \Delta_ya=(N-1)|x-y|^{-1}$ and  
$$\Delta_x^3 a = \Delta_y^3a =\left\{\begin{array}{ll}
C\delta(x-y),&\mbox{if}\quad N=5;\\
3(N -1)(N - 3)(N - 5)|x-y|^{-5},&\mbox{if}\quad N> 5,
\end{array}\right.$$
when $N =5$, choosing $u_j = v_j,$ we get
$$ \displaystyle\sum_{j=1}^m \displaystyle\int_0^T \displaystyle\int_{\R^5} |u_j(x,t)|^4\,dx\,dt\lesssim \displaystyle\sup_{[0,T]}|M_a^{\otimes_2}|.$$
If $N>5$, it follows that
$$ \displaystyle\sum_{j=1}^m \displaystyle\int_0^T\displaystyle\int_{\R^{N}\otimes \R^{N}}\frac{|u_j(x,t)|^2|u(y,t)|^2}{|x - y|^5}\,dx\,dy\,dt\lesssim \displaystyle\sup_{[0,T]}|M_a^{\otimes_2}|. $$
This finishes the proof.



\begin{thebibliography}{99}


\bibitem{Adams}{\bf R. Adams}, {\em Sobolev spaces}, Academic. New York, (1975).

\bibitem{Artzi} {\bf M. Ben-Artzi, H. Koch and J. C. Saut}, {\em Dispersion estimates for fourth order Schr\"odinger equations}, C. R. Math. Acad.
Sci. S\'er. 1. Vol. 330, 87-92, (2000).
\bibitem{Cazenave} {\bf T. Cazenave}, {\em Semilinear Schr\"odinger Equations}, Courant Lect Notes Math, Vol. 10, Univ Pierre et Marie Curie, (2003).

\bibitem{Cazenave 1} {\bf T. Cazenave}, {\em An Introduction to Nonlinear Schr\"odinger Equations}, Textos Met. Mat. 26, Instituto de Matematica UFRJ, Rio de Janeiro, (1996).

\bibitem{Colliander} {\bf J. Colliander, M. Grillakis and N. Tzirakis}, {\em Tensor products and correlation estimates with applications to nonlinear 
Schr\"odinger equations}, Comm. Pure Appl. Math. Vol. 62, no.1, 920-968, (2009).

\bibitem{Fibich} {\bf G. Fibich, B. Ilan and G. Papanicolaou}, {\em Self-focusing with fourth-order dispersion}, SIAM J. Appl. Math, Vol. 62, no. 4, 1437-1462, (2002).

\bibitem{Guo} {\bf B. Guo and B. Wang}, {\em The global Cauchy problem and scattering of solutions for nonlinear Schr\"odinger equations in $H^s$},  Differential Integral Equations, Vol. 15, no. 9, 1073-1083, (2002).

\bibitem{km}{\bf C. Kenig and F. Merle}, {\em Global well-posedness, scattering and blow-up for the energy-critical focusing non-linear wave equation}, Acta Math. Vol. 201, no. 2, 147-212, (2008).

\bibitem{Merle}{\bf C. E. Kenig and F. Merle},{\em Global wellposedness, scattering and blow up for the energy critical, focusing, nonlinear Schr\"odinger equation in the radial case}. Invent. Math. Vol. 166, 645-675, (2006).

\bibitem{Hasegawa}{\bf A. Hasegawa and F. Tappert},{\em Transmission of stationary nonlinear optical pulses in dispersive dielectric fibers II. Normal dispersion}, Appl. Phys. Lett. Vol. 23, 171-172, (1973).
 
\bibitem{Karpman}{\bf V. I. Karpman}, {\em Stabilization of soliton instabilities by higher-order dispersion: fourth-order nonlinear Schr\"odinger equation}, Phys. Rev. E. Vol. 53, no. 2, 1336-1339, (1996).

\bibitem{Karpman 1}{\bf V.I. Karpman and A. G. Shagalov}, {\em Stability of soliton described by nonlinear Schr\"odinger type equations with higher-order dispersion}, Phys D. Vol. 144, 194-210, (2000).

\bibitem{Levandosky} {\bf S. Levandosky and W. Strauss},{\em Time decay for the nonlinear Beam equation}, Meth. Appl. Anal. Vol. 7, 479-488, (2000).

\bibitem{Lions}{\bf P.L. Lions}, {\em Sym\'etrie et compacit\'e dans les espaces de Sobolev}, J. Funct. Anal. Vol. 49, no. 3, 315-334, (1982).

\bibitem{Nghiem}{\bf Nghiem V. Nguyen, Rushun Tian, Bernard Deconinck, and Natalie Sheils}, {\em Global existence for a coupled system of Schr\"odinger equations with power-type nonlinearities}, J. Math. Phys. Vol. 54, 011503, (2013).

\bibitem{Miao}{\bf C. Miao, G. Xu and L. Zhao}, {\em Global well-posedness and scattering for the focusing energy-critical nonlinear Schr\"odinger equations of fourth-order in the radial case}, J. D. E. Vol. 246, 3715-3749, (2009).

\bibitem{Miao 1}{\bf C. Miao, G. Xu and L. Zhao}, {\em Global well-posedness and scattering for the defocusing energy-critical nonlinear Schr\"odinger equations of fourth-order in dimensions $d \geq 9$}, J. D. E. Vol. 251, no. 12, 15, 3381-3402, (2011).

\bibitem{Miao 2}{\bf C. Miao, H. Wu and J. Zhang}, {\em Scattering theory below energy for the cubic fourth-order Schr\"odinger equation}, published online in J. Math. Nachr. doi: 10.1002/mana.201400012.

\bibitem{Pausader}{\bf B. Pausader}, {\em Global well-posedness for energy critical fourth-order Schr\"odinger equations in the radial case}, Dyn.
Partial Differ. Equ. Vol. 4, no. 3, 197-225, (2007).
\bibitem{Pausader 1 }{\bf B. Pausader}, {\em The cubic fourth-order Schr\"odinger equation}, J. F. A, Vol. 256, 2473-2517, (2009).

\bibitem{Pausader 2}{\bf B. Pausader}, {\em The focusing energy-critical fourth-order Schr\"odinger equation with radial data}, Discrete Contin. Dyn.
Syst. Ser. A, Vol. 24, no. 4, 1275-1292, (2009).

\bibitem{ts}{\bf T. Saanouni}, {\em A note on fourth-order nonlinear Schr\"odinger equation}, Ann. Funct. Anal. Vol. 6, no. 1, 249-266, (2015).



\bibitem{W.A.S}{\bf W. A. Strauss}, {\em Nonlinear wave equations}, CBMS. Lect. A. M. S, {73}, (1989).

\bibitem{Swanson}{\bf Ch. A. Swanson}, {\em The best Sobolev constant},  Appl. Anal. Vol. 47, 227-239, (1992).

\bibitem{Tao} {\bf T. Tao}, {\em Nonlinear dispersive equations: local and global analysis}, CBMS Reg. Ser. Math. (2006).

\bibitem{Zakharov} { \bf V. E. Zakharov}, {\em Stability of periodic waves of finite amplitude on the surface of a deep fluid}. Sov. Phys. J. Appl. Mech. Tech. Phys. Vol. 4, 190-194, (1968).





\end{thebibliography}
\end{document}